\documentclass[12pt,draftclsnofoot,onecolumn]{IEEEtran}

\usepackage{psfrag,graphics,color,epsfig,subfigure}
\usepackage{amsmath,amssymb,amscd, mathrsfs}
\usepackage{enumerate,afterpage}
\usepackage{pifont}
\allowdisplaybreaks

\newcommand{\cB}			{{\mathscr{B}}}
\newcommand{\SP}[1]		{{\cB}_{r}^{{#1}}}
\newcommand{\SPone}		{{\SP{1}}}
\newcommand{\SPtwo}		{{\SP{2}}}
\newcommand{\SPthree}		{{\SP{3}}}

\newcommand{\cW}			{{\mathscr{W}}}
\newcommand{\ts}[1]		{{\textstyle{#1}}}
\newcommand{\demi}		{{\ts{\frac{1}{2}}}}
\newcommand{\op}[1]		{{\mathcal{{#1}}}}

\newcommand{\ssf}[2]		{{{#1}_{\mathsf{{#2}}}}}

\newcommand{\done}{\ding{182}}
\newcommand{\dtwo}{\ding{183}}
\newcommand{\dthree}{\ding{184}}
\newcommand{\dfour}{\ding{185}}
\newcommand{\dfive}{\ding{186}}

\newcommand{\opBmp}[1]	{{\op{B}_{k,i}}}

\newcommand{\fundterm}		{\Psi}

\newcommand{\eps}{\varepsilon}

\newcommand{\Z}{\mathbb{Z}}
\newcommand{\R}{\mathbb{R}}
\newcommand{\ds}{\displaystyle}

\newcommand{\be}{\begin{equation}}
\newcommand{\ee}{\end{equation}}
\newcommand{\bea}{\begin{eqnarray}}
\newcommand{\eea}{\end{eqnarray}}
\newcommand{\ba}{\begin{array}}
\newcommand{\ea}{\end{array}}

\newcommand{\er}[1]{\hbox{(\ref{#1})}}

\newcommand{\nn}			{{\nonumber}}

\newtheorem{theorem}            {Theorem}[section]

\newtheorem{lemma}              [theorem]{Lemma}
\newtheorem{sideremark}         [theorem]{Remark}
\newtheorem{sidenote}           [theorem]{Note}
\newtheorem{sideeg}           [theorem]{Example}
\newtheorem{sideconj}           [theorem]{Conjecture}
\newtheorem{sideassumption}   [theorem]{Assumption}

\newenvironment{remark}         {\begin{sideremark}\rm}{\end{sideremark}}

\newenvironment{assumption} {\begin{sideassumption}\it}{\end{sideassumption}}
\newenvironment{proof}		{{\it Proof:}}{\hfill{$\blacksquare$}}

\begin{document}

%%%%%%%%%%%%%%%%%%%%%%%%%%%%%%%%%%%%%%%%%%%%%%%%%%%%%%%%%%%%%%%%%%%%%
%%
%%		Title and Abstract
%%

\title{\LARGE \bf A max-plus based fundamental solution for a class of discrete time linear regulator problems }

\author{Huan Zhang${}^{\dagger}$ ~~~~~~~~ Peter M. Dower${}^{\dagger}$
\thanks{${}^{\dagger}$ Zhang and Dower are with the Department of Electrical and Electronic Engineering, University of Melbourne, Melbourne, Victoria 3010, Australia. Email: {\tt \{hzhang5,pdower\}@unimelb.edu.au.} This research is supported by grants FA2386-12-1-4084 and DP120101549 from AFOSR and the Australian Research Council.}
}

\maketitle

\begin{abstract}
Efficient Riccati equation based techniques for the approximate solution of discrete time linear regulator problems are restricted in their application to problems with quadratic terminal payoffs. Where non-quadratic terminal payoffs are required, these techniques fail due to the attendant non-quadratic value functions involved. In order to compute these non-quadratic value functions, it is often necessary to appeal directly to dynamic programming in the form of grid- or element-based iterations for the value function. These iterations suffer from poor scalability with respect to problem dimension and time horizon. In this paper, a new max-plus based method is developed for the approximate solution of discrete time linear regulator problems with non-quadratic payoffs. This new method is underpinned by the development of new fundamental solutions to such linear regulator problems, via max-plus duality.
In comparison with a typical grid-based approach, a substantial reduction in computational effort is observed in applying this new max-plus method. A number of simple examples are presented that illustrate this and other observations.
\end{abstract}

%%%%%%%%%%%%%%%%%%%%%%%%%%%%%%%%%%%%%%%%%%%%%%%%%%%%%%%%%%%%%%%%%%%%%
%%
%%		Introduction
%%

\section{Introduction}
After more than 40 years of study, the ``linear quadratic regulator problem'' (or LQR problem) remains ubiquitous in the field of optimal control \cite{AM:89}, \cite{CM:70}. Given a specific linear time invariant system, quadratic running payoff, and terminal payoff, the objective of the LQR (optimal control) problem is to determine a control sequence that (when applied to the linear system in question) maximizes the aggregated running and quadratic terminal payoffs over a specific (possibly infinite) time horizon. It is well known that the value function defined by the LQR problem is quadratic. The Hessian of this quadratic value function is either the solution of a difference (or differential) Riccati equation (DRE) in the finite horizon case, or the stabilizing solution of an algebraic Riccati equation (ARE) in the infinite horizon case. Solutions to either equation, and hence the corresponding LQR problem, can be computed very accurately and efficiently using existing numerical tools (for example, {\sf MATLAB$^{\textsf{TM}}$}).

Both the DRE and ARE encode invariance of the space of quadratic functions (defined on the state space) with respect to the dynamic programming evolution operator associated with a quadratic running payoff and linear dynamics. Consequently, both equations are restricted in their application to problems with quadratic terminal payoffs. Where the terminal payoff employed is non-quadratic, the DRE / ARE solution path for the corresponding linear regulator problem is inherently invalid (as the corresponding value function involved need not be quadratic). Instead, it is necessary to appeal directly to the dynamic programming principle to obtain an iteration for the value function. This iteration is in general infinite dimensional, regardless of the state dimension. Consequently, approximate value function iterations employing state-space grids, basis functions, etc, arise out of necessity, but remain intrinsically limited in their application due to the curse-of-dimensionality. Consequently, where the time horizon is long or the state dimension high, the approximate solution of a linear regulator problem in the company of a non-quadratic terminal payoff remains a computationally expensive (and sometimes even prohibitive) exercise.

In this paper, a new computational method is developed for approximating the value function associated with a class of discrete time linear regulator problems in which the terminal payoff is non-quadratic. Motivated by recent related work \cite{M:08,DM1:11,DM1:12,DM2:12}, this new method relies on the development of a max-plus based fundamental solution for the class of linear regulator problems of interest. Using max-plus duality arguments \cite{AGL:08,BCOQ:94,CGQ:99,M:03,M:06,M:07,M:08}, this fundamental solution captures the behaviour of the associated dynamic programming evolution operator, and is independent of the terminal payoff employed. By applying this fundamental solution to the terminal payoff associated with a specific linear regulator problem, the attendant value function (and hence the solution of this linear regulator problem) may be computed. Furthermore, by appealing to the algebraic structure of the fundamental solution, a substantial improvement in computation time relative to grid-based iterative methods can be achieved. This improvement is demonstrated via a number of computational examples. In addition, the limiting behaviour of finite horizon linear regulator problems is investigated via the fundamental solutions presented. While value functions associated with non-quadratic terminal payoffs are typically non-quadratic on finite horizons, it is shown (under mild conditions) that these converge to quadratic value functions in the infinite horizon. There, the effect of a non-quadratic terminal payoff is shown to reduce to an additive offset in this infinite horizon limit. The convergence results employed generalize well known DRE / ARE results \cite{AM:89,BGP85,CM:70}. Preliminary results by the authors have recently been reported in \cite{ZD1:13, ZD2:13}.

In terms of organization, Section \ref{sec:LQR} describes the linear regulator problem and associated max-plus vector spaces of interest. Section \ref{sec:fundamental} derives the §max-plus fundamental solutions and discusses their properties. Section \ref{sec:inf} discusses the infinite horizon linear regulator problem with non-quadratic terminal payoff. Examples are given in Section \ref{sec:exam} to demonstrate the computational advantages of the proposed method. Section \ref{sec:conc} is a brief conclusion. Throughout, $\Z_{\ge 0}$ and $\Z_{>0}$ are used to denote the non-negative and positive integers respectively. $\R^-\doteq \R\cup\{-\infty\}$ is used to denote the extended reals, while $\R^{n}$ denotes $n$-dimensional Euclidean space equipped with the standard $2$-norm denoted by $|\cdot|$. $\ssf{\lambda}{min}(A)$ and $\ssf{\lambda}{max}(A)$ denote respectively the smallest and largest eigenvalue of matrix $A\in\R^{n\times n}$. $I\in\R^{n\times n}$ and $\op{I}$ are used to denote the $n$ by $n$ identity matrix and an identity operator respectively.

%%%%%%%%%%%%%%%%%%%%%%%%%%%%%%%%%%%%%%%%%%%%%%%%%%%%%%%%%%%%%%%%%%%%%
%%
%%		Class of linear regulator problems, dynamic programming, and max-plus duality
%%

\section{Linear regulator problems with non-quadratic payoff}
\label{sec:LQR}

\subsection{Optimal control problem}
Throughout, attention is restricted to discrete-time time invariant linear systems of the form
\begin{align}
	x_{k+1} & = A\, x_k+B\, w_k\,, \quad x_0=x,
	\label{eq:system}
\end{align}
in which $x_k\in\R^n$ and $w_k\in\R^m$ denote the state and input, both at time $k\in\Z_{\ge 0}$, and $x\in\R^n$ denotes the initial state. $A\in\R^{n\times n}$ and $B\in \R^{n\times m}$ denote constant matrices with real-valued entries. The following properties concerning \er{eq:system} are assumed to hold throughout.
\begin{assumption}
\label{ass:system}
(i) $[A, B]$ is controllable; and (ii) $\text{rank}(B)=m\le n$.
\end{assumption}
The value function $W_K:\R^n\rightarrow\R$ of a linear regulator problem defined on time horizon $K\in\Z_{\ge 0}$ is given by
\begin{align}
	W_K(x)
	& \doteq
	\sup_{w\in\cW[0,K-1]} J_K(x;\, w)\,,
	\label{eq:value}
\end{align}
in which $\cW[0,K-1] \doteq (\R^m)^K$ denotes the attendant space of input sequences with indices in $[0,K-1]\cap\Z$, and {$J_K:\R^n\times\cW[0,K-1]\rightarrow\R$} denotes the total (accumulated running plus terminal) payoff
\begin{align}
	J_K(x;w)\doteq\sum_{k=0}^{K-1}\left(\ts{\frac{1}{2}}x_k^T\, \Phi \, x_k - \ts{\frac{\gamma^2}{2}}\, |w_k|^2 \right) + \Psi(x_K)\,,
	\label{eq:payoff}
\end{align}
in which $w_k\in\R^m$ denotes the $k^{th}$ element of sequence $w\in\cW[0,K-1]$, and $x_k$ denotes the corresponding element of the state sequence generated by \er{eq:system} subject to this input sequence. The running payoff in \er{eq:payoff} is parameterized by $\Phi\in\R^{n\times n}$ (a symmetric and positive definite real-valued matrix, i.e. $\Phi=\Phi^T>0$), and a gain parameter $\gamma\ge0$. The terminal payoff is denoted by the function $\Psi:\R^n\rightarrow\R$.

\begin{remark}
\label{rmk:W0}
Note that by convention, $W_0(x)=\Psi(x),\,\,x\in\R^n$.
\end{remark}

\subsection{Non-quadratic payoffs, attendant max-plus vector spaces, and duality}
The class of optimal control problems described above (and of interest in this paper) is further restricted to those with non-quadratic terminal payoffs that enjoy a quadratic upper bound. In formalizing this assumption, and in the subsequent development of a max-plus based solution to this class of problems, it is convenient to define a hierarchy of three function spaces. In particular, define $\SPone\subset\SPtwo\subset\SPthree$ as
\begin{equation}
	\begin{aligned}
		\SPone & \doteq \left\{ \phi\in\SPtwo \, \biggl| \, \phi \text{ is convex} \right\}\,,
		\\
		\SPtwo & \doteq \left\{ \phi\in\SPthree \, \biggl| \, \phi \text{ is semi-convex} \right\}\,,
		\\
		\SPthree & \doteq \left\{ \phi : \R^n\rightarrow\R^- \, \biggl|\, \exists \ c\in\R\text{ s.t. } \phi(x) \le \ts{\frac{r}{2}} \, |x|^2 + c\text{ for all } x\in\R^n \right\}\,.
	\end{aligned}
	\label{eq:SP-all}
\end{equation}
% The class of optimal control problem of interest is formally restricted to those problems with terminal payoff $\Psi:\R^n\rightarrow\R^-$ satisfying the following assumption:
\begin{assumption}
\label{ass:growth}
There exists an $r\in\R$ such that the terminal payoff $\Psi$ in \er{eq:payoff} satisfies $\Psi\in\SPthree$.
\end{assumption}
In view of \er{eq:SP-all}, recall that a max-plus based fundamental solution for a class of continuous time LQR problems with finite dimensional dynamics was formulated and developed in \cite{M:08} for terminal payoffs in the space $\SPtwo$. (Related infinite dimensional extensions have also been developed, see \cite{DM1:11,DM2:12,DM1:12}.) In the spirit of that work, it may be shown that the function spaces \er{eq:SP-all} are all max-plus vector spaces (see for example \cite{M:06}). In particular, $a\otimes \phi_1\oplus \phi_2\in\SP{i}$ for all $a\in\R^-$, $\phi_{1,2}\in\SP{i}$, and $i\in{1,2,3}$, where the binary operations $\oplus$ and $\otimes$ denote max-plus addition and multiplication, viz
\begin{align}
	a \oplus b & \doteq \max(a,\, b)\,, \quad
	a \otimes b \doteq a + b\,.
	\nn
\end{align}
The max-plus integral of $\phi\in\SP{i}$ is similarly defined as $\int_{\R^n}^\oplus \phi(x)\, dx \doteq \sup_{x\in\R^n} \phi(x)$. With a view to employing primal-dual relationships defined with respect to each of these spaces, it is convenient to define three corresponding functions $\psi^i(\cdot,z)\in\SP{i}$, parametrized by $z\in\R^n$, as
\begin{equation}
	\begin{aligned}
		\psi^1(x,z)
		 &\doteq z^T x,\quad
		\\
		\psi^2(x,z)
		&\doteq -\demi\, (x-z)^T \, M\, (x-z),\quad
		\\
		\psi^3(x,z)
		 &\doteq \delta(x-z).
	\end{aligned}
	\label{eq:psi-all}
\end{equation}
Here, $M = M^T\in\R^{n\times n}$ is positive definite, and $\delta:\R^n\rightarrow\R^-$ denotes the extended real-valued indicator function defined by
$
\delta(\xi)
	 \doteq \left\{ \begin{aligned}
		0\,, && \xi = 0\,,
		\\
		-\infty\,, && \xi \ne 0\,.
	\end{aligned} \right.
$
As mentioned, these functions $\psi^i$ of \er{eq:psi-all} may be used to define primal-dual relationships with respect to each  function space $\SP{i}$. In particular, for any $\phi\in\SP{i}$, it may be noted that the primal $\phi$ and dual $a$ are related via
\begin{align}
	\phi
	& = \op{D}_{\psi^i}^{-1} \, a\,, \quad
	a = \op{D}_{\psi^i} \, \phi\,,
\end{align}
where $\psi^i$ is as per \er{eq:psi-all}, and $\op{D}_{\psi^i}$, $\op{D}_{\psi^i}^{-1}$ denote respectively the dual and inverse dual (with respect to function $\psi^i\in\SP{i}$) defined by
\begin{align}
	\op{D}_{\psi^i} \, \phi
	& = \left( \op{D}_{\psi^i} \, \phi \right)(\cdot)
	\doteq -\int_{\R^n}^\oplus \psi^i(x,\cdot) \otimes \left( - \phi(x) \right)\, dx\,,
	\label{eq:op-dual}
	\\
	\op{D}_{\psi^i}^{-1} \, a
	& =  \left( \op{D}_{\psi^i}^{-1} \, a \right)(\cdot)
	\doteq \int_{\R^n}^\oplus \psi^i(\cdot,z) \otimes a(z) \, dz\,.
	\label{eq:op-inv-dual}
\end{align}
By inspection of \er{eq:psi-all}, $\op{D}_{\psi^1}$ is the well-known convex dual, while $\op{D}_{\psi^2}$ is the semi-convex dual employed in finite dimensions in \cite{FM:00,M:03,M:04,M:06,M:07,M:08}, and in infinite dimensions in \cite{DM1:11,DM2:12,DM1:12}. $\op{D}_{\psi^3}$ can be verified directly as $(\op{D}_{\psi^3}\, \phi)(z) = -\max_{x\in\R^n}\left\{\delta(x-z)-\phi(x)\right\}=\phi(z)$. That is, the max-plus dual (with respect to $\psi^3\in\SP{3}$) of any function in $\SP{3}$ is itself. For these duality operators $\op{D}_{\psi^i}$ and $\op{D}_{\psi^i}^{-1}$ of \er{eq:op-dual} and \er{eq:op-inv-dual} to be well defined for the fundamental solutions in Section \ref{sec:fundamental} (see Remark \ref{Rmk:exis}), the following assumptions regarding the basis functions \er{eq:psi-all} are posed. 
\begin{assumption}
\label{ass:existence}
(i=1) $P_k^{-1}$ exists for all $k\in\Z_{>0}$, where $P_k$ satisfies the difference Riccati equation (DRE)
\begin{align}
P_{k+1} = \Phi+A^T\, P_k\, A+A^T \, P_k \, B \left( \gamma^2\, I - B^T\, P_k\, B \right)^{-1} B^T\, P_k\, A\,,
\label{eq:DRE}
\end{align}
with $P_0 = 0$.

(i=2) There exists an $M = M^T\in\R^{n\times n}$, $M>0$, such that $P_k + M > 0$ for all $k\in\Z_{>0}$, where $P_k$ satisfies the DRE \er{eq:DRE} with $P_0 = -M$.
\end{assumption}

%%	Dynamic programming

\subsection{Dynamic programming}
A standard application of dynamic programming (see, for example, \cite{B:05}) to the optimal control problem defined by \er{eq:value} yields a (one-step) dynamic programming principle for the finite horizon value function $W_{k}:\R^n\rightarrow\R$ indexed by time horizon $k\in\Z_{\ge 0}$. In particular
\begin{equation}
	W_{k+1} = \op{S}_1 \, W_k\,, \quad W_0 = \Psi,
	\label{eq:DPP}
\end{equation}
where $\op{S}_1$ denotes the (one-step) dynamic programming evolution operator defined by
\begin{align}
	(\op{S}_1 \, \phi)(x)
	= \left(\op{S}_1^{\Phi,\gamma}\, \phi\right)(x)
	& \doteq
	\sup_{w\in\R^m} \left\{{\frac{1}{2}} x^T\, \Phi\, x-{\frac{\gamma^2}{2}} \, |w|^2 + \phi(A\, x + B\, w) \right\}.
	\label{eq:op-DPP-1}
\end{align}
(Superscript notation $\op{S}_1^{\Phi,\gamma}$ will be used where convenient to emphasize the explicit dependence on $\Phi$ and $\gamma$.) Where the terminal payoff $\Psi:\R^n\rightarrow\R$ is a quadratic function of the form $\Psi(x) = \demi\, x^T\, \Lambda\, x$ (with $\Lambda=\Lambda^T\ge0$, $\Lambda\in\R^{n\times n})$, the value function $W_k:\R^n\rightarrow\R$ is also a quadratic function, with $W_k(x)=\demi\, x^T\, P_k\, x$ for all $k\in\Z_{\ge 0}$. As \er{eq:DPP} holds for all $x\in\R^n$, the value function iteration defined by \er{eq:DPP} with respect to the time horizon $k\in\Z_{\ge 0}$ immediately reduces to DRE \er{eq:DRE} with $P_0=\Lambda$. This DRE describes a finite dimensional representation for the potentially infinite dimensional iteration \er{eq:DPP}. The key to the reduced order representation \er{eq:DRE} of \er{eq:DPP} is the fact that the space of quadratic functions is invariant with respect to the dynamic programming evolution operator $\op{S}_1$ of \er{eq:op-DPP-1}. Where the terminal payoff is a non-quadratic function, this invariance cannot be exploited. That is, DRE \er{eq:DRE}  need not hold.

The definition \er{eq:op-DPP-1} of the one-step dynamic programming evolution operator $\op{S}_1$ may be extended to the $(k+1)$-step case, $k\in\Z_{>0}$, via the recursion
\begin{align}
	\op{S}_{k+1} \, \phi
	& = \op{S}_1 \left( \op{S}_k\, \phi \right)
	= \op{S}_1\, \op{S}_k\, \phi\,.
	\label{eq:op-DPP-k}
\end{align}
\begin{remark}
\label{rmk:semigroup-1}
By convention (see Remark \ref{rmk:W0}), define $\op{S}_0\doteq \op{I}$. \er{eq:op-DPP-k} implies that the time indexed set of operators $\{\op{S}_k, k\in\Z_{\ge0}\}$ satisfies the property $\op{S}_{k_1+k_2}=\op{S}_{k_1}\op{S}_{k_2}, k_1,k_2\in\Z_{\ge0}$. Hence, this set of operators forms a semigroup.
\end{remark}

The value function $W_k$ of \er{eq:value} may accordingly be expressed in terms of the terminal cost $\Psi$ and $\op{S}_k$ via $W_k = \op{S}_k\, \Psi$, c.f. \er{eq:DPP}. Invariance of the max-plus vector spaces $\SP{i}$ of \er{eq:SP-all} with respect to this family of evolution operators is key to the subsequent development of a max-plus based solution to the optimal control problem of \er{eq:value}.
\begin{theorem}
\label{thm:Sk-invariant}
Suppose $\ssf{\lambda}{max}(A^TA)<1$. Then, for any given $i\in\{1,2,3\}$, $r\in\R_{>0}$, there exist $\Phi_0\in\R^{n\times n}$, $\Phi_0\ge0$, and $\gamma_0\in\R_{>0}$ such that for all $\Phi\in\R^{n\times n}$, $\Phi\le\Phi_0$, $\gamma\in\R_{>0}$, $\gamma\ge \gamma_0$,
\begin{align}
	\Psi\in\SP{i}
	& \quad \Longrightarrow \quad
	\op{S}_{k}\, \Psi \equiv
	\op{S}_{k}^{\Phi,\gamma}\, \Psi\ \in\ \SP{i}
	\label{eq:Sk-invariant}
\end{align}
for all $k\in\Z_{>0}$.
\end{theorem}
\if{false}
The case for $i=2$ is proved in \cite{M:06}. The proofs for $i=1$ and $i=3$ are analogous and omitted for brevity.
\fi
\begin{proof}
First consider the case where $i=3$. In order to show that $\SPthree$ is invariant as per \er{eq:Sk-invariant}, an induction argument is applied. To this end, suppose that $\Psi\in\SPthree$, that is, there exists $c\in\R$ such that $\Psi(x) \le \ts{\frac{r}{2}} \, |x|^2 + c\text{ for all } x\in\R^n $. Applying \er{eq:op-DPP-1},
\begin{align}
	\left(\op{S}_1^{\Phi,\gamma}\, \Psi\right) (x)
	& = \sup_{w\in\R^m} \left\{\ts{\frac{1}{2}}\,x^T\, \Phi\, x - \ts{\frac{\gamma^2}{2}}\, |w|^2 + \Psi(Ax+Bw)\right\}
	\nn\\
	& \le \sup_{w\in\R^m}\left\{\ts{\frac{1}{2}}\,x^T\, \Phi\, x- \ts{\frac{\gamma^2}{2}}\, |w|^2 + \ts{\frac{r}{2}}\, |Ax+Bw|^2+c\right\}
	\nn\\
	& = \demi\, x^T\,  \Xi^{\Phi,\gamma}_r\, x +c
	\le \demi\, \ssf{\lambda}{max}\left( \Xi^{\Phi,\gamma}_r\right)\, |x|^2 +c\,,
	\label{eq:S-1-bound}
\end{align}
where
\begin{align}
	x^T\,\Xi^{\Phi,\gamma}_r\,x
	& \doteq \sup_{w\in\R^m}\left\{\ts{\frac{1}{2}}\,x^T\, \Phi\, x- \ts{\frac{\gamma^2}{2}}\, |w|^2 + \ts{\frac{r}{2}}\, |Ax+Bw|^2\right\}
	\nn\\
	& = x^T\,\left( \Phi+r\,A^TA+r^2\, A^T\, B\left( \gamma^2\, I - r\, B^T\, B\right)^{-1} B^T\, A\right)\,x\,.
	\nn
\end{align}
Select $\Phi_0\in\R^{n\times n}$ positive semi-definite such that
\begin{align}
0<\ssf{\lambda}{max}(\Phi_0)\le \ts{\frac{r}{3}}(1-\ssf{\lambda}{max}(A^TA))
\label{eq:Phi-0-select}
\end{align}
and $\gamma_0\in\R_{>0}$ such that
\begin{equation}
	\gamma_0^2 \, I
	\ge 2\, r\, B^T\, B\,,
	\quad
	\gamma_0^2 \, I
	\ge \ts{\frac{r^2}{\ssf{\lambda}{max}(\Phi_0)}} A^T\, B \, B^T\, A\,.
	\label{eq:gamma-0-select}
\end{equation}
(By inspection, note that such a $\gamma_0$ and $\Phi_0$ always exist.) Hence, for any $\Phi\in\R^{n\times n}$, $0\le \Phi\le \Phi_0$, and $\gamma\in\R_{>0}$, $\gamma\ge \gamma_0$,  the left-hand inequality of \er{eq:gamma-0-select} implies that
\begin{align}
	\gamma^2\, I - r\, B^T\, B
	& \ge \gamma_0^2\, I - r\, B^T\, B
	\ge \gamma_0^2\, I - \ts{\frac{\gamma_0^2}{2}} \, I
	= \ts{\frac{\gamma_0^2}{2}} \, I\,.
	\nn
\end{align}
Consequently, $\gamma^2\, I - r\, B^T\, B$ is invertible, with $(\gamma^2\, I - r\, B^T\, B)^{-1} \le \ts{\frac{2}{\gamma_0^2}}\, I$. Furthermore, by definition of $\Xi^{\Phi,\gamma}_r$,
\begin{align}
\label{eq:Hessian-bound}
	\Xi^{\Phi,\gamma}_r
	& =  \Phi +r\,A^TA+ r^2\, A^T\, B\left( \gamma^2\, I - r\, B^T\, B\right)^{-1} B^T\, A
\\\nn
	 &\le  \ssf{\lambda}{max}( \Phi )\,I+\ssf{\lambda}{max}(A^TA)\, I + \ts{\frac{2\, r^2}{\gamma_0^2}} A^T\, B\, B^T\, A
\\\nn
 &\le \ssf{\lambda}{max}( \Phi_0 )\, I +\ssf{\lambda}{max}(A^TA)\, I+2\, \ssf{\lambda}{max}( \Phi_0 )\, I
 \le r\, I\,,
\end{align}
where the third and fourth inequalities follow by the inequalities of \er{eq:gamma-0-select} and \er{eq:Phi-0-select} respectively. Hence, \er{eq:S-1-bound} yields that for any $\Phi\in\R^{n\times n}$, $\gamma\in\R_{>0}$ such that $\Phi\le \Phi_0$, $\gamma\ge\gamma_0$,
\begin{align}
	\left(\op{S}_1^{\Phi,\gamma}\, \Psi\right) (x)\le \ts{\frac{r}{2}} \, |x|^2+c
\end{align}
holds for all $x\in\R^n$. That is, $\op{S}_1^{\Phi,\gamma}\, \Psi\in\SP{3}$, so the stated assertion holds for $k=1$. In order to show that it also holds for any $k\in\Z_{>0}$, suppose that $\op{S}_k^{\Phi,\gamma}\, \Psi\in\SP{3}$, that is, there exists $\bar{c}\in\R$ such that $\op{S}_k^{\Phi,\gamma}(x) \le \ts{\frac{r}{2}} \, |x|^2 + \bar{c}\text{ for all } x\in\R^n $. Then, applying \er{eq:op-DPP-k} followed by \er{eq:op-DPP-1},
\begin{align}
	\left(\op{S}_{k+1}^{\Phi,\gamma} \, \Psi\right)(x)
	& = \left( \op{S}_1^{\Phi,\gamma} \, \op{S}_k^{\Phi,\gamma} \, \Psi \right)(x)
\le  \sup_{w\in\R^m}\left\{\ts{\frac{1}{2}}\,x^T\, \Phi\, x- \ts{\frac{\gamma^2}{2}}\, |w|^2 + \left(\op{S}_k^{\Phi,\gamma} \, \Psi\right)\left( A\, x + B\, w \right)
	\right\}
	\nn\\
	& \le \sup_{w\in\R^m}\left\{\ts{\frac{1}{2}}\,x^T\, \Phi\, x- \ts{\frac{\gamma^2}{2}}\, |w|^2 + \ts{\frac{r}{2}}\, |Ax+Bw|^2 +\bar{c}\right\} = \demi\, x^T\, \Xi^{\Phi,\gamma}_r \, x +\bar{c} \le \ts{\frac{r}{2}} \, |x|^2 +\bar{c}\,,
	\nn
\end{align}
where the last inequality follows by \er{eq:Hessian-bound}. That is, $\op{S}_{k+1}^{\Phi,\gamma}\, \Psi\in\SP{3}$. Hence, by induction, the stated assertion holds for $i=3$.

In order to show that the stated assertion holds for $i\in\{1,2\}$, inspection of \er{eq:SP-all} and the fact that $\SPone\subset\SPtwo\subset\SPthree$ reveals that it only remains to be shown that $\op{S}_1^{\Phi,\gamma}$ preserves convexity and semiconvexity (respectively). The fact that semiconvexity is preserved is well-known, see for example Theorem 4.9 on page 67 in \cite{M:06}. The convex case is included to illustrate the arguments involved. In particular, fix any $x_{1,2}\in\R^n$, $\lambda\in(0,1)$, and $\phi\in\SPone$. Then, by convexity of $\phi$, and semi-positiveness property of $\Phi\ge0$
\begin{align}
	\left( \op{S}_1^{\Phi,\gamma}\, \phi \right)&(\lambda\, x_1  + (1-\lambda)\, x_2)
\nn\\
&= \sup_{w\in\R^m} \left\{ \ba{c}\ts{\frac{1}{2}}\,(\lambda\, x_1  + (1-\lambda)\, x_2)^T\, \Phi\, (\lambda\, x_1  + (1-\lambda)\, x_2) \\- \ts{\frac{\gamma^2}{2}} \, |w|^2 + \phi\left( A(\lambda \, x_1+(1-\lambda)\, x_2) + B\, w \right)\ea\right\}
	\nn\\
	&=\sup_{w\in\R^m} \left\{ \ba{c}\ts{\frac{1}{2}}\,(\lambda\, x_1  + (1-\lambda)\, x_2)^T\, \Phi\, (\lambda\, x_1  + (1-\lambda)\, x_2) \\- \ts{\frac{\gamma^2}{2}} \, |w|^2 + \phi\left (\lambda(Ax_1+Bw)+(1-\lambda)(Ax_2+Bw) \right)\ea\right\}
	\nn\\
	& \le \sup_{w\in\R^m} \left\{\ba{c} \ts{\frac{\lambda}{2}}\,x_1^T\, \Phi\, x_1+\ts{\frac{(1-\lambda)}{2}}x_2^T\,\Phi\,x_2 - \ts{\frac{\gamma^2}{2}} \, |w|^2 + \\ \lambda\, \phi(A\, x_1 + B\, w)+ (1-\lambda) \, \phi(A\, x_2 + B\, w) \ea\right\}
	\nn\\
	& \le \ba{c}\lambda\, \sup_{w\in\R^m}\left\{\ts{\frac{1}{2}}\, x_1^T \, \Phi\, x_1 - \ts{\frac{\gamma^2}{2}} \, |w|^2 + \phi(A\, x_1 + B\, w)\right\}
	\\
	+ (1-\lambda)\, \sup_{w\in\R^m}\left\{ \ts{\frac{1}{2}}\,x_2^T\, \Phi\, x_2 -\ts{\frac{\gamma^2}{2}} \, |w|^2 + \phi(A\, x_2 + B\, w)\right\}\ea
	\nn\\
	& = \lambda \left( \op{S}_1^{\Phi,\gamma} \, \phi \right) (x_1) + (1-\lambda) \left( \op{S}_1^{\Phi,\gamma} \, \phi \right) (x_2)\,.
	\nn
\end{align}
Hence, $\op{S}_1^{\Phi,\gamma}\, \phi$ is convex, thereby demonstrating that $\op{S}_1^{\Phi,\gamma}\, \phi\in\SPone$.
\end{proof}

The max-plus linearity of the $k$-step dynamic programming evolution operators $\op{S}_k$ of \er{eq:op-DPP-1} does not depend on the specific max-plus linear space $\SP{i}, i\in\{1,2,3\}$. The case of $i=2$ is proved in Theorem 4.5 on page 66 of \cite{M:06}.
\begin{lemma}
\label{lem:linear}
The $k$-step dynamic programming evolution operator $\op{S}_k$ of \er{eq:op-DPP-k} is max-plus linear for all $k\in\Z_{>0}$. That is, for all
for all $a\in\R^-$, $\phi,\,\theta\in\SP{i}$, $i\in\{1,2,3\}$, and $k\in\Z_{>0}$,
\begin{align}
	\op{S}_k \, (a\otimes \phi \oplus \theta)
	& = a\otimes (\op{S}_k\, \phi) \oplus (\op{S}_k \, \theta)\,.
	\label{eq:linear}
\end{align}
\end{lemma}

\if{false}

\begin{proof}
Fix any $a\in\R^-$, $\phi,\,\theta\in\SP{i}$, $i\in\{1,2,3\}$, $x\in\R^n$, by inspection of \er{eq:op-DPP-1}, \begin{align}
	( \op{S}_1 & \, (a\otimes \phi \oplus \theta) )(x)
	= \sup_{w\in\R^m} \left\{\ts{\frac{1}{2}}\, x^T \, \Phi\, x - \ts{\frac{\gamma^2}{2}} \, |w|^2 + (a\otimes \phi \oplus \theta)(A\, x + B\ w) \right\}
	\nn\\
	& = \sup_{w\in\R^m} \left\{ \ts{\frac{1}{2}}\,x^T \, \Phi\, x - \ts{\frac{\gamma^2}{2}} \, |w|^2 + \max\left( a + \phi(A\, x + B\ w),\, \theta(A\, x + B\ w) \right) \right\}
	\\
	& = \sup_{w\in\R^m} \max \left( a + \ts{\frac{1}{2}}\,x^T \, \Phi\, x - \ts{\frac{\gamma^2}{2}} \, |w|^2 +  \phi(A\, x + B\ w),\,
						\ts{\frac{1}{2}}\,x^T \, \Phi\, x - \ts{\frac{\gamma^2}{2}} \, |w|^2 + \theta(A\, x + B\ w) \right)
	\nn\\
	& = \max\left( a + (\op{S}_1 \, \phi)(x),\, (\op{S}_1\, \theta)(x) \right) = a\otimes (\op{S}_1\, \phi)(x) \oplus  (\op{S}_1\, \theta)(x)\,.
\nn
\end{align}
That is, $\op{S}_1$ is max-plus linear. Suppose that $\op{S}_k$ is max-plus linear for $k\in\Z_{>0}$. Applying \er{eq:op-DPP-k},
\begin{align}
	\op{S}_{k+1}\, ( a\otimes\phi\oplus\theta )
	& =
	\op{S}_1 \, \op{S}_k ( a\otimes\phi\oplus\theta )= \op{S}_1 \, (a\otimes (\op{S}_k\, \phi) \oplus (\op{S}_k\, \theta) )
	\nn\\
	& = a\otimes (\op{S}_1\, \op{S}_k\, \phi) \oplus (\op{S}_1\, \op{S}_k\, \theta) = a\otimes (\op{S}_{k+1}\, \phi) \oplus (\op{S}_{k+1}\, \theta)\,,
	\nn
\end{align}
where the second equality is by hypothesis, and the third is by max-plus linearity of $\op{S}_1$. Hence, by induction, $\op{S}_k$ is max-plus linear for all $k\in\Z_{>0}$.
\end{proof}

\fi

%%%%%%%%%%%%%%%%%%%%%%%%%%%%%%%%%%%%%%%%%%%%%%%%%%%%%%%%%%%%%%%%%%%%%
%%
%%		Fundamental solution semigroups
%%

\section{Max-plus fundamental solution and computational method}
\label{sec:fundamental}

\subsection{Max-plus fundamental solution semigroup}
\label{subsec:semigroup}

Where the terminal payoff $\Psi$ is non-quadratic, the value function $W_k$ \er{eq:value} may be computed via grid-based dynamic programming iterations \er{eq:DPP} for $ k\in\Z_{\ge0}$ \cite{KD:92}. However, this method is computationally expensive for problems with higher state dimensions, due to the exponential increase in grid points required to represent the state space. This is the well-known curse-of-dimensionality \cite{M:06}. By exploiting the max-plus linearity of the operator $\op{S}_k, k\in\Z_{\ge0}$, a more efficient computational method can be developed. This method employs an analogous max-plus fundamental solution to that developed in \cite{M:08}. To this end, define a set of auxiliary value functions $\mathrm{S}_{k,i}:\R^n\times\R^n\rightarrow\R^-, k\in\Z_{\ge0}, i\in\{1,2,3\}$, by
\begin{align}
\label{eq:kernel-primal}
\mathrm{S}_{k,i}(x,z)\doteq\left(\op{S}_k\psi^i(\cdot,z)\right)(x), \hspace{.5cm} \forall~ (x,z)\in\R^n\times\R^n.
\end{align}
Applying the definition of $\op{D}_{\psi^i}$ and $\op{D}_{\psi^i}^{-1}$ in \er{eq:op-dual} and \er{eq:op-inv-dual}, and the max-plus linearity of $\op{S}_k, k\in\Z_{\ge0}$, from Lemma \ref{lem:linear} yields
\begin{align}
\label{eq:fund-solu}
W_k(x)&=\left(\op{S}_k\Psi\right)(x)=\left(\op{S}_k\int_{\R^n}^\oplus\psi^i(\cdot,z)\otimes(\op{D}_{\psi^i}\Psi)(z)\,dz\right)(x)
\\\nn
&=\int_{\R^n}^\oplus\left(\op{S}_k\psi^i(\cdot,z)\right)(x)\otimes(\op{D}_{\psi^i}\Psi)(z)\,dz=\int_{\R^n}^\oplus\mathrm{S}_{k,i}(x,z)\otimes(\op{D}_{\psi^i}\Psi)(z)\,dz.
\end{align}
Hence, the value function $W_k$ can be computed by performing a max-plus integration of the max-plus product of $\mathrm{S}_{k,i}$ of \er{eq:kernel-primal} and the dual of the terminal payoff $\op{D}_{\psi^i}\Psi$. The function $\mathrm{S}_{k,i}$ of \er{eq:kernel-primal} is independent of the terminal payoff $\Psi$. When $\mathrm{S}_{k,i}$ is computed, it can be used to compute any value function $W_k$ corresponding to an arbitrary terminal payoff $\Psi$ via \er{eq:fund-solu}. From \er{eq:kernel-primal}, the function $\mathrm{S}_{k,i}$ is obtained by applying the dynamic programming evolution operator $\op{S}_k$ of \er{eq:op-DPP-k} to the functions $\psi^i$ of \er{eq:psi-all}. As a consequence of the linear dynamics \er{eq:system}, quadratic running payoff in \er{eq:payoff} and the quadratic basis function $\psi^i\in\SP{i}$ used as the terminal payoff, the function $\mathrm{S}_{k,i}$ is the value function of an LQR problem  \cite{AM:89}. Hence it is quadratic of the form
\begin{align}
\mathrm{S}_{k,i}(x,z)=\frac{1}{2}\left[\ba{c}x \\ z \ea\right]^TQ_{k,i} \left[\ba{c}x \\ z \ea\right]=\frac{1}{2}\left[\ba{c}x \\ z \ea\right]^T  \left[\ba{cc}Q_{k,i}^{11}&Q_{k,i}^{12}\\Q_{k,i}^{21}&Q_{k,i}^{22}\ea\right]   \left[\ba{c}x \\ z \ea\right],
\label{eq:quad-fund-primal}
\end{align}
where $Q_{k,i}\in(\R^-)^{2n\times 2n}$. An iterative representation for the Hessian follows by dynamic programming. These iterations can be written down independently of the initial conditions $Q_{1,i}, i\in\{1,2,3\}$. These initial conditions are derived separately in Section \ref{sec:ini}.
\begin{theorem}
\label{thm:dynamics-Q}
The Hessian $Q_{k,i}\, k\in\Z_{\ge0}, i\in\{1,2,3\}$ of $\mathrm{S}_{k,i}$ in \er{eq:quad-fund-primal} satisfy
\begin{align}
\nn
Q_{k+1,i}^{11}&=\Phi+A^TQ_{k,i}^{11}A+A^TQ_{k,i}^{11}B(\gamma^2I-B^TQ_{k,i}^{11}B)^{-1}B^TQ_{k,i}^{11}A,
\\\nn
Q_{k+1,i}^{12}&=A^TQ_{k,i}^{12}+A^TQ_{k,i}^{11}B(\gamma^2I-B^TQ_{k,i}^{11}B)^{-1}B^TQ_{k,i}^{12},
\\\label{eq:dynamics-Q}
Q_{k+1,i}^{21}&=(Q_{k+1,i}^{12})^T,
\\\nn
Q_{k+1,i}^{22}&=Q_{k,i}^{22}+Q_{k,i}^{21}B(\gamma^2I-B^TQ_{k,i}^{11}B)^{-1}B^TQ_{k,i}^{12}.
\end{align}
\end{theorem}
\begin{proof}
Applying the quadratic form \er{eq:quad-fund-primal} of $\mathrm{S}_{k,i}$ along with the definitions of the operators $\op{S}_1$ and $\op{S}_k$ in \er{eq:op-DPP-1} and \er{eq:op-DPP-k} respectively yields
\begin{align}
\nn
\mathrm{S}_{k+1,i}(x,z)&=\frac{1}{2}\left[\ba{c}x \\ z \ea\right]^TQ_{k+1,i} \left[\ba{c}x \\ z \ea\right]=\left(\op{S}_{k+1}\psi^i(\cdot,z)\right)(x)
\\\nn
                        &=\left(\op{S}_1\op{S}_k\psi^i(\cdot,z)\right)(x)=\left(\op{S}_1\left(\frac{1}{2}\left[\ba{c}\cdot \\ z \ea\right]^TQ_{k,i} \left[\ba{c}\cdot \\ z \ea\right]\right)\right)(x)
\\\nn
                        &=\sup_{w\in\R^m}\left\{{\frac{1}{2}}\,x^T\Phi x-\frac{1}{2}\gamma^2\,|w|^2+\frac{1}{2}\left[\ba{c}Ax+Bw \\ z \ea\right]^TQ_{k,i} \left[\ba{c}Ax+Bw \\ z \ea\right]\right\}.
\end{align}
The argument of the supremum on the right-hand side is quadratic in $w$, and consequently, the maximisation can be performed analytically by completion of squares. In particular, the supremum is achieved by $w^\ast=(\gamma^2\,I-B^TQ_{k,i}^{11}B)^{-1}(B^TQ_{k,i}^{11}A^Tx+B^TQ_{k,i}^{12}z)$.  Iteration \er{eq:dynamics-Q} follows by explicitly evaluating the supremum using $w^\ast$.
\end{proof}

For each $i\in\{1,2,3\}$, the functions $\mathrm{S}_{k,i},k\in\Z_{\ge0}$ can be propagated forward to $\mathrm{S}_{k+1,i}$ via the iteration \er{eq:dynamics-Q}. As shown in the continuous time \cite{M:08} and infinite dimensional cases \cite{DM1:11}, \cite{DM2:12}, \cite{DM1:12}, it is more efficient to compute $\mathrm{S}_{k,i},k\in\Z_{\ge0}$, for longer time horizons via their max-plus duals, as a specific time horizon doubling technique can be developed. To this end, let $\mathrm{B}_{k,i}(\cdot,z):\R^n\rightarrow\R^-, z\in\R^n$, denote the max-plus dual of $\mathrm{S}_{k,i}(\cdot,z):\R^n\rightarrow\R^-, z\in\R^n$, with respect to $\psi^i\in\SP{i}, i\in\{1,2,3\}$, so that by \er{eq:op-dual}
\begin{align}
\label{eq:dual-fund-2}
\mathrm{B}_{k,i}(y,z)\doteq\left(\op{D}_{\psi^i}\mathrm{S}_{k,i}(\cdot,z)\right)(y)=-\int_{\R^n}^\oplus\psi^i(x,y)\otimes(-\mathrm{S}_{k,i}(x,z))\,dx.
\end{align}
The function $\mathrm{S}_{k,i}$ is recovered from $\mathrm{B}_{k,i}$ via the inverse dual operator $\op{D}_{\psi^i}^{-1}$ of \er{eq:op-inv-dual}, with
\begin{align}
\label{eq:dual-fund-1}
\mathrm{S}_{k,i}(x,z)=\left(\op{D}^{-1}_{\psi^i}\mathrm{B}_{k,i}(\cdot,z)\right)(x)=\int_{\R^n}^\oplus\psi^i(x,y)\otimes\mathrm{B}_{k,i}(y,z)\,dy.
\end{align}
The functions $\mathrm{B}_{k,i}, k\in\Z_{\ge0}$ of \er{eq:dual-fund-2} can be interpreted as kernels in defining max-plus integral operators $\op{B}_{k,i}, k\in\Z_{\ge0}$, on spaces $\SP{i}, i\in\{1,2,3\}$, via
\begin{align}
\label{eq:mp-semigroup}
\left(\op{B}_{k,i}{a}\right)(y)\doteq\int_{\R^n}^{\oplus}\mathrm{B}_{k,i}(y,z)\otimes a(z)\, dz.
\end{align}
\begin{remark}
\label{rmk:B0}
Fix any $i\in\{1,2,3\}$, $\mathrm{S}_{0,i}=\psi^i$  from the definition \er{eq:kernel-primal} of $\mathrm{S}_{k,i}$. Hence $\mathrm{B}_{0,i}(y,z)=\delta(y-z)$ from \er{eq:dual-fund-2} and subsequently, $\op{B}_{0,i}=\op{I}$ according to \er{eq:mp-semigroup}.
\end{remark}

The operators $\op{B}_{k,i}$ are closely related to the operators $\op{S}_{k,i}$ from \er{eq:op-DPP-1} and \er{eq:op-DPP-k} via  $\op{D}_{\psi^i}$ of \er{eq:op-dual} and $\op{D}_{\psi^i}^{-1}$ of \er{eq:op-inv-dual}.
\begin{theorem}
\label{thm:operator-duality}
For any $k\in\Z_{\ge0}$ and $i\in\{1,2,3\}$,
\begin{align}
\label{eq:operator-duality}
\op{S}_k=\op{D}_{\psi^i}^{-1}\,\op{B}_{k,i}\,\op{D}_{\psi^i}.
\end{align}
\end{theorem}
\begin{proof}
Fix any $\phi\in\SP{i}$, $i\in\{1,2,3\}$, $k\in\Z_{>0}$ and  $x\in\R^n$. Applying \er{eq:fund-solu}, the definition\er{eq:dual-fund-2}  of $\mathrm{B}_{k,i}$,   and the duality operators $\op{D}_{\psi^i}$ and $\op{D}_{\psi^i}^{-1}$ of \er{eq:op-dual} and \er{eq:op-inv-dual},
\begin{align}
\nn
\left(\op{S}_k\phi\right)(x)&=\int_{\R^n}^{\oplus}\left(\op{D}_{\psi^i}\phi\right)(z)\otimes\mathrm{S}_{k,i}(x,z)\,dz=\int_{\R^n}^{\oplus}\left(\op{D}_{\psi^i}\phi\right)(z)\otimes\left(\op{D}_{\psi^i}^{-1}\mathrm{B}_{k,i}(\cdot,z)\right)(x)\,dz
\\\nn
      &=\int_{\R^n}^{\oplus}\left(\op{D}_{\psi^i}\phi\right)(z)\otimes\left(\int_{\R^n}^\oplus\psi^i(x,y)\otimes\mathrm{B}_{k,i}(y,z)\,dy\right)\,\,dz
\\\nn
      &=\int_{\R^n}^{\oplus}\psi^i(x,y)\otimes\left(\int_{\R^n}^{\oplus}\mathrm{B}_{k,i}(y,z)\otimes\left(\op{D}_{\psi^i}\phi\right)(z)\,dz\,\right)\,dy
\\\nn
      &=\int_{\R^n}^{\oplus}\psi^i(x,y)\otimes \left(\op{B}_{k,i}\op{D}_{\psi^i}\phi\right)(y)\,dy=\left(\op{D}_{\psi^i}^{-1}\,\op{B}_{k,i}\,\op{D}_{\psi^i}\phi\right)(x).
\end{align}
\end{proof}
\begin{remark}
\label{rmk:semigroup-2}
From Remark \ref{rmk:semigroup-1}, $\{\op{S}_k, k\in\Z_{\ge0}\}$ is a semigroup. Theorem \ref{thm:operator-duality}  implies that for any $k_1,k_2\in\Z_{\ge0}$
$
\op{B}_{k_1+k_2,i}=\op{D}_{\psi^i}\op{S}_{k_1+k_2}\op{D}^{-1}_{\psi^i}=\op{D}_{\psi^i}\op{S}_{k_1}\op{S}_{k_2}\op{D}^{-1}_{\psi^i}=\op{D}_{\psi^i}\op{S}_{k_1}\op{D}^{-1}_{\psi^i}\op{D}_{\psi^i}\op{S}_{k_2}\op{D}^{-1}_{\psi^i}=\op{B}_{k_1,i}\op{B}_{k_2,i}.
$
It is also shown in Remark \ref{rmk:B0} that $\op{B}_{0,i}=\op{I}$. Thus, the operators $\{\op{B}_{k,i}, k\in\Z_{\ge0}\}$ form a semigroup.
\end{remark}

The value functions $W_k, k\in\Z_{\ge0}$, of \er{eq:value} are propagated by the semigroup $\{\op{S}_k,k\in\Z_{\ge0}\}$ via \er{eq:DPP}, or equivalently, $W_k=\op{S}_k{W_0}$. From Theorem \ref{thm:operator-duality},
$
W_k=\op{D}_{\psi^i}^{-1}\,\op{B}_{k,i}\,\op{D}_{\psi^i}{W_0},
$
which can be equivalently expressed as $\op{D}_{\psi^i}W_k=\op{B}_{k,i}\,(\op{D}_{\psi^i}{W_0}), k\in\Z_{>0}$. Thus, the semigroup $\{\op{B}_{k,i}, k\in\Z_{\ge0}\}$ propagates the max-plus dual of the value functions $\op{D}_{\psi^i}W_k$. Consequently, there are two paths obtaining the value function $W_k$ from the initial condition (terminal payoff) $W_0=\Psi$ as shown in panel (a) of Figure \ref{fig:propagations}.

\begin{figure}[h]
\begin{center}
\begin{equation}
\ba{cc}
\begin{CD}
		\hspace{1cm}\fundterm 	@>\text{via}~\op{S}_{k}~\text{of \er{eq:op-DPP-k}}>>	\op{S}_{k}\, \fundterm &\hspace{4cm}  Q_{1,i}	 @>\text{via \er{eq:dynamics-Q}}>>  Q_{k,i}
\\
		\hspace{1cm}@VV\op{D}_{\psi^i}V		@AA\op{D}_{\psi^i}^{-1}A   & \hspace{-4.5cm} @VV\Gamma^iV		 \hspace{-0.5cm}@AA\Gamma^i A
\\
		\hspace{1cm}\op{D}_{\psi^i}\, \fundterm	 @>\text{via}~\opBmp{k}~\text{of \er{eq:mp-semigroup}}>>	\opBmp{k}\, \op{D}_{\psi^i}\, \fundterm &\hspace{4cm}  \Theta_{1,i}	 @>\text{via \er{eq:B-update}}>>	 \Theta_{k,i}
	\end{CD}
\\
\hspace{-5.5cm}\ba{c}\text{\footnotesize (a): Propagation of  $W_k$ via $\op{S}_k$ of \er{eq:op-DPP-k} or via $\op{B}_{k,i}$ of \er{eq:mp-semigroup}.}\ea &\hspace{-5cm} \ba{c}\text{\footnotesize (b): Propagation of $Q_{k,i}$ of \er{eq:quad-fund-primal} and $\Theta_{k,i}$ of \er{eq:quad-fund-dual}.}\ea 	
\nn
\ea
\end{equation}
\end{center}
\caption{Propagation of value functions via two semigroups and propagation of matrices $Q_{k,i}, \Theta_{k,i}$.}
\label{fig:propagations}
\end{figure}

\subsection{Propagation of the fundamental solution semigroup kernels}
\label{subsec:kernel}

The propagation of the fundamental solution semigroup $\{\op{B}_{k,i}, k\in\Z_{\ge0}\}$ can be represented by the evolution of its kernel functions $\mathrm{B}_{k,i}, k\in\Z_{\ge0}, i\in\{1,2,3\}$ of \er{eq:dual-fund-2}.
\begin{theorem}
\label{thm:expon}
For $(y,z)\in\R^n\times\R^n, k_1,k_2\in\Z_{\ge0}$
\begin{align}
\label{eq:B2k}
\mathrm{B}_{k_1+k_2,i}(y,z)=\int_{\R^n}^\oplus\mathrm{B}_{k_1,i}(y,\rho)\otimes\mathrm{B}_{k_2,i}(\rho,z)\,d\rho.
\end{align}
\end{theorem}
This iteration does not depend on the choice of max-plus vector space $\SP{i}$. It has the same form as in the continuous time \cite{M:08} and infinite dimensional cases \cite{DM1:11},  \cite{DM2:12}, \cite{DM1:12}. The proof of Theorem \ref{thm:expon} follows as per \cite{M:08} and is omitted for brevity.

According to Theorem \ref{thm:dynamics-Q}, as $\mathrm{S}_{k,i}$ takes quadratic form with \er{eq:quad-fund-primal}, it can be shown that the kernel $\mathrm{B}_{k,i}$ of \er{eq:dual-fund-2} is also with that quadratic form, with
\begin{align}
\label{eq:quad-fund-dual}
\mathrm{B}_{k,i}(y,z)=-\frac{1}{2}\left[\ba{c}y \\ z \ea\right]^T \Theta_{k,i}\left[\ba{c}y \\z\ea\right]=-\frac{1}{2}\left[\ba{c}y \\ z \ea\right]^T \left[\ba{cc}\Theta_{k,i}^{11}&\Theta_{k,i}^{12}\\\Theta_{k,i}^{21}&\Theta_{k,i}^{22}\ea\right]
\left[\ba{c}y \\ z \ea\right].
\end{align}
Hence, iterations \er{eq:B2k} are reduced to iterations on the matrices $\Theta_{k,i}, k\in\Z_{\ge0}$. These iterations are specified by a matrix operation  $\Omega_1\circledast \Omega_2$ defined by
\begin{align}
\label{eq:circledast}
\Omega_1\circledast \Omega_2\doteq\left[\ba{cc}\Omega_1^{11} &0\\0&\Omega_2^{22} \ea\right]-\left[\ba{c}\Omega_1^{12}\\ \Omega_2^{21} \ea\right](\Omega_1^{22}+\Omega_2^{11})^{-1}[\Omega_1^{21}~~\Omega_2^{12}].
\end{align}
Here, $\Omega_j\in\R^{2n\times 2n}, \Omega_j=\Omega_j^T, j=1,2 $, satisfy $\Omega_1^{22}+\Omega_2^{11}>0$.
\begin{theorem}\label{thm:exponential}
Suppose that $\mathrm{B}_{k_j}^i$ for $j=1,2$ and $i\in\{1,2,3\}$ are quadratic of the form \er{eq:quad-fund-dual} with $\Theta_{k_j,i}$.
%\begin{align}
%$\mathrm{B}_{k_j,i}(y,z)=-\textstyle{\frac{1}{2}\left[\ba{c}y \\ z \ea\right]^T \Theta_{k_j,i}\left[\ba{c}y \\ z \ea\right]}
%=-\frac{1}{2}\left[\ba{c}y \\ z \ea\right]^T \left[\ba{cc} \Theta_{k_j,i}^{11} &\Theta_{k_j,i}^{12}\\\Theta_{k_j,i}^{21} &\Theta_{k_j,i}^{22} \ea\right]\left[\ba{c}y \\ z \ea\right].
%.$
%\end{align}
Then, $\mathrm{B}_{k_1+k_2,i}$ is quadratic of the form \er{eq:quad-fund-dual} with $\Theta_{k_1+k_2,i}$ given by
%\begin{align}
%$
%\mathrm{B}_{k_1+k_2,i}(y,z)=-\frac{1}{2}\left[\ba{c}y \\ z \ea\right]^T \Theta_{k_1+k_2,i}\left[\ba{c}y \\ z \ea\right]
%$
\begin{align}
\label{eq:B-update}
\Theta_{k_1+k_2,i}=\Theta_{k_1,i}\circledast\Theta_{k_2,i}.
\end{align}
\end{theorem}
Theorem \ref{thm:exponential} has the same form for all spaces $\SP{i}, i\in\{1,2,3\}$. The proof follows as per \cite{M:08}, where it is proved for the case where $i=2$. The proofs for the remaining cases follow similarly, and are omitted for brevity.

Equation \er{eq:B-update} implies that the evolution of kernels $\mathrm{B}_{k,i}$ need not involve every time index $k\in\Z_{\ge0}$. Indeed, any sequence of time indices may be employed, provided that each element of that sequence can be expressed as a sum of two prior (smaller) elements. Time index doubling is one obvious example. In that case, by generating a sequence  $(\mathrm{B}_{1,i}, \mathrm{B}_{2^1,i}, \mathrm{B}_{2^2,i}, \cdots, \mathrm{B}_{2^l,i})$ for $l\in\Z_{>0}$ using equation \er{eq:B-update}, only $l$ matrix operations $\circledast$ are required to propagate $\Theta_{1,i}$ to $\Theta_{2^l,i}$. This is the key motivation behind  computing the auxiliary value functions $\mathrm{S}_{k,i}$ of \er{eq:kernel-primal} via the propagation of the kernels $\mathrm{B}_{k,i}$ of \er{eq:dual-fund-2}. However, in the computation of $\mathrm{S}_{k,i}$ via $\mathrm{B}_{k,i}$, two additional steps are required. Firstly, at the initial time $k=1$, it is necessary to compute the dual $\mathrm{B}_{1,i}(\cdot,z)=\op{D}_{\psi^i}\mathrm{S}_{1,i}(\cdot,z)$ of the initial auxiliary value function $\mathrm{S}_{1,i}$ according to \er{eq:dual-fund-2}. Secondly, at the final time $k$, the function $\mathrm{S}_{k,i}$ must be recovered via $\mathrm{S}_{k,i}=\op{D}^{-1}_{\psi^i}\mathrm{B}_{k,i}(\cdot,z)$ according to \er{eq:dual-fund-1}. It will be shown next that these maximization operations \er{eq:dual-fund-2} and \er{eq:dual-fund-1} are reduced to a matrix operation specified by $\Gamma^i:\R^{2n\times2n}\rightarrow\R^{2n\times2n}, i\in\{1,2,3\}$, where
\begin{align}
\label{eq:Gamma}
\nn
\Gamma^1(Q)&\doteq\left[\ba{cc} (Q^{11})^{-1} &-(Q^{11})^{-1}Q^{12}\\-Q^{21}(Q^{11})^{-1} &Q^{21}(Q^{11})^{-1}Q^{12}-Q^{22} \ea\right],
\\
\Gamma^2(Q)&\doteq\left[\ba{cc} M(Q^{11}+M)^{-1}M-M &-M(Q^{11}+M)^{-1}Q^{12}
\\ -Q^{21}(Q^{11}+M)^{-1}M &Q^{21}(Q^{11}+M)^{-1}Q^{12}-Q^{22} \ea\right],
\\\nn
\Gamma^3(Q)&\doteq-Q.
\end{align}
Here, the matrix $M$ in the definition of $\Gamma^2$ is the Hessian used to define the quadratic basis functions $\psi^2$ of \er{eq:psi-all} in space $\SPtwo$. It is required that $Q^{11}>0$ in the definition of $\Gamma^1$ and $Q^{11}+M>0$ in the definition of $\Gamma^2$ in order for the respective inverses to exist. It can be verified directly that
$
Q=\Gamma^i(\Gamma^i(Q))\doteq\Gamma^i\circ\Gamma^i(Q),
$ or $\Gamma^i\circ\Gamma^i=\op{I}$.
\begin{remark}
\label{Rmk:exis}
For $i=1,2$, by inspection of \er{eq:DRE} with \er{eq:dynamics-Q}, if $Q_{1,1}^{11}$ and $Q_{1,2}^{11}$ are as per \er{eq:Q11} and \er{eq:Q12}, respectively, then Assumption \ref{ass:existence} states that $Q_{k,1}^{11}$ is invertible and $Q^{11}_{k,2}  + M >0$ for all $k\in\Z_{>0}$. In that case, the matrix operations $\Gamma^i$ of \er{eq:Gamma} are well defined for all $Q_{k,i}$ $k\in\Z_{>0}$.
\end{remark}

\begin{theorem}
\label{thm:Q-Theta}
For any $k\in\Z_{>0}$, suppose that $\mathrm{S}_{k,i}$ of \er{eq:kernel-primal} and $\mathrm{B}_{k,i}$ of \er{eq:dual-fund-2} are quadratics of the form \er{eq:quad-fund-primal} and \er{eq:dual-fund-2}, respectively. Then, $Q_{k,i}$ and $\Theta_{k,i}$ are related via $\Gamma^i$ of \er{eq:Gamma} by
\begin{align}
\label{eq: Q-Theta}
\Theta_{k,i}=\Gamma^i (Q_{k,i}),~~~
Q_{k,i}=\Gamma^i(\Theta_{k,i}).
\end{align}
\end{theorem}
\begin{proof}
 From \er{eq:dual-fund-2} and the definition \er{eq:Gamma} of $\Gamma^i$,
\begin{align}
\nn
\mathrm{B}_{k,i}(y,z)&=-\int_{\R^n}^\oplus\psi^i(x,y)\otimes(-\mathrm{S}_{k,i}(x,z))\,dx
=-\max_{x\in\R^n}\left\{\psi^i(x,y)+(-\mathrm{S}_{k,i}(x,z))\right\}
\\\nn
                      &=-\max_{x\in\R^n}\left\{\psi^i(x,y)-\frac{1}{2}\left[\ba{c}x \\ z \ea\right]^T\left[\ba{cc}Q_{k,i}^{11}&Q_{k,i}^{12}\\Q_{k,i}^{21}&Q_{k,i}^{22}\ea\right] \left[\ba{c}x \\ z \ea\right]  \right\}
%\\\nn
                      =-\frac{1}{2}\left[\ba{c}y \\ z \ea\right]^T\Gamma^i(Q_{k,i})\left[\ba{c}y \\ z \ea\right].
\end{align}
Comparing with \er{eq:quad-fund-dual} yields $\Theta_{k,i}=\Gamma^i(Q_{k,i})$.
On the other hand, from \er{eq:dual-fund-1},
\begin{align}
\nn
\mathrm{S}_{k,i}(x,z)&=\int_{\R^n}^\oplus\psi^i(x,y)\otimes\mathrm{B}_{k,i}(y,z)\,dy=\max_{y\in\R^n}\left\{\psi^i(x,y)\otimes\mathrm{B}_{k,i}(y,z)\right\}
%\\\nn
%&=\max_{y\in\R^n}\left\{\psi^i(x,y)-\frac{1}{2}\left[\ba{c}y \\ z \ea\right]^T \left[\ba{cc}\Theta_{k,i}^{11}&\Theta_{k,i}^{12}\\\Theta_{k,i}^{21}&\Theta_{k,i}^{22}\ea\right] \left[\ba{c}y \\ z \ea\right]\right\}
\\\nn
                         &=\max_{y\in\R^n}\left\{\psi^i(y,x)-\frac{1}{2}\left[\ba{c}y \\ z \ea\right]^T\left[\ba{cc}\Theta_{k,i}^{11}&\Theta_{k,i}^{12}\\\Theta_{k,i}^{21}&\Theta_{k,i}^{22}\ea\right] \left[\ba{c}y \\ z \ea\right]  \right\}
%\\\nn
                      =\frac{1}{2}\left[\ba{c}x \\ z \ea\right]^T\Gamma^i(\Theta_{k,i})\left[\ba{c}x \\ z \ea\right],
\end{align}
 where the property $\psi^i(x,y)=\psi^i(y,x), i\in\{1,2,3\}$, is used. Comparing with \er{eq:quad-fund-primal} yields $Q_{k,i}=\Gamma^i(\Theta_{k,i})$.
\end{proof}

The propagations of $Q_{k,i}$ and $\Theta_{k,i}$ for $k\in\Z_{>0}$ are shown in panel (b) in Figure \ref{fig:propagations}.

\subsection{Initializations}
\label{sec:ini}

The initializations of iterations \er{eq:dynamics-Q} for $Q_{k,i}, k\in\Z_{>0}$ and \er{eq:B-update} for $\Theta_{k,i}, k\in\Z_{>0}$ depend on the specific spaces $\SP{i}, i\in\{1,2,3\}$.

{\it{For space $\SPone$}}: According to \er{eq:kernel-primal}, the function $\mathrm{S}_{1,1}$ is
\begin{align}
\nn
\mathrm{S}_{1,1}(x,z)&=\left(\op{S}_1\psi^1(\cdot,z)\right)(x)=\sup_{w\in\R^m}\left\{\ts{\frac{1}{2}}\,x^T\Phi x-\ts{\frac{1}{2}}\gamma^2w^Tw+z^T(Ax+Bw)\right\}
\\\label{eq:Q11}
                     &=\frac{1}{2}\left[\ba{c} x \\ z\ea\right]^T\,Q_{1,1}\,\left[\ba{c} x \\ z\ea\right],\quad\text{with}\quad Q_{1,1}=\left[\ba{cc} \Phi & A^T \\ A & \gamma^{-2}BB^T\ea\right].
\end{align}
Thus,
\begin{align}
\label{eq:Theta-1-1}
\Theta_{1,1}=\Gamma^1(Q_{1,1})=\left[\ba{cc} \Phi^{-1} &-\Phi^{-1}A^T \\ -A\Phi^{-1} & A\Phi^{-1}A^T-\gamma^{-2}BB^T \ea\right].
\end{align}

{\it{For space $\SPtwo$}}: According to \er{eq:kernel-primal}, the function $\mathrm{S}_{1,2}$ is
\begin{align}
\nn
\mathrm{S}_{1,2}(x,z)&=\left(\op{S}_1\psi^2(\cdot,z)\right)(x)=\sup_{w\in\R^m}\left\{\ba{l}\ts{\frac{1}{2}}\,x^T\Phi x-\frac{1}{2}\gamma^2w^Tw\\-\frac{1}{2}(Ax+Bw-z)^TM(Ax+Bw-z)\ea\right\}
\\\label{eq:Q12}
                    &=\frac{1}{2}\left[\ba{c} x \\ z\ea\right]^TQ_{1,2}\left[\ba{c} x \\ z\ea\right], ~\text{with}~ Q_{1,2}=\left[\ba{cc}Q_{1,2}^{11}&Q_{1,2}^{12}\\Q_{1,2}^{21}&Q_{1,2}^{22}\ea\right]=\left[\ba{cc} A^T\Delta A+\Phi &-A^T\Delta \\-\Delta A &\Delta \ea\right],
\end{align}
where
$
\Delta=MB(\gamma^2I+B^TMB)^{-1}B^TM-M.
$
Thus,
\begin{align}
\label{eq:Theta-1-2}
\Theta_{1,2}=\Gamma^2(Q_{1,2})=\left[\ba{cc} M(M+Q_{1,2}^{11})^{-1}M-M & -M(M+Q_{1,2}^{11})^{-1}Q_{1,2}^{12} \\-Q_{1,2}^{21}(M+Q_{1,2}^{11})^{-1}M & Q_{1,2}^{21}(M+Q_{1,2}^{11})^{-1}Q_{1,2}^{12}-Q_{1,2}^{22}\ea\right].
\end{align}

{\it{For space $\SPthree$}}: In this case, the max-plus dual of any $\phi\in\SPthree$ is itself, that is, $\left(\op{D}_{\psi^3}\phi\right)(z)=\phi(z)$ from \er{eq:op-inv-dual}.
 From the definition of $\mathrm{S}_{k,3}$ of \er{eq:kernel-primal}
\begin{align}
\label{eq:value-constr}
\mathrm{S}_{k,3}(x,z)&=\left(\op{S}_k\psi^3(\cdot,z)\right)(x)=\sup_{w\in\cW_{[0,k-1]}}\left\{\sum_{i=0}^{k-1}\left(\ts{\frac{1}{2}}\,x_i^T\Phi x_i-\ts{\frac{1}{2}}\gamma^2w_i^Tw_i\right)+\delta(x_k-z)\right\}
\\\nn
                          &=\sup_{w\in\cW_{[0,k-1]}}\left\{\sum_{i=0}^{k-1}\left(\ts{\frac{1}{2}}\,x_i^T\Phi x_i-\ts{\frac{1}{2}}\gamma^2w_i^Tw_i\right)\biggl| x_k=z\right\}.
\end{align}
That is, $\mathrm{S}_{k,3}(x,z)$ is the optimal control problem \er{eq:value} with constraints $x_0=x$ and $x_k=z$. To compute the constrained optimal control problem \er{eq:value-constr}, denote
\begin{align}
\label{eq:Lamdak}
\Lambda_k(x,z)\doteq\left\{w\in\cW_{[0,k-1]}\biggl|x_0=x, x_k=z~\text{subject to }\er{eq:system}\right\}
\end{align}
the set of controls $w=(w_{0},w_1,\cdots,w_{k-1})$ that steers the initial state from $x_0=x$ to final state $x_k=z$. It is necessary that $\Lambda_k(x,z)\neq\emptyset$ for the function $\mathrm{S}_{k,3}(x,z)$ of \er{eq:value-constr} to be quadratic on $\R^n\times\R^n$. By definition \er{eq:Lamdak}, the set $\Lambda_k$ of controls is intimately tied to reachability via the matrix $B$. Consequently, in characterizing the initialization $\mathrm{S}_{1,3}$ in terms of set $\Lambda_k$, a number of specific cases for the dimensions of matrix $B\in\R^{n\times m}$ must be considered in view of Assumption \ref{ass:system}.

{\it{1) $m=n$}:} In this case, $B\in\R^{n\times m}$ is invertible since rank$(B)=m$ from Assumption \ref{ass:system}. Hence
$
\Lambda_1(x,z)=\left\{w_0\in\R^m\,|\,w_0=B^{-1}(Ax-z)\right\}.
$
From \er{eq:value-constr},
\begin{align}
\nn
\mathrm{S}_{1,3}(x,z)&=\ts{\frac{1}{2}}\,x^T\Phi x-\ts{\frac{1}{2}}\,\gamma^2w_0^Tw_0=\ts{\frac{1}{2}}\,x^T\Phi x-\ts{\frac{1}{2}}\,\gamma^2(Ax-z)^T(BB^T)^{-1}(Ax-z)
\\\label{eq:Q-1-3}
&=\frac{1}{2}\left[\ba{c}x\\z\ea\right]^T Q_{1,3}\left[\ba{c}x\\z\ea\right],\quad \text{with}\quad Q_{1,3}= \left[\ba{cc} \Phi-\gamma^2A^T(BB^T)^{-1}A&\gamma^2A^T(BB^T)^{-1}\\\gamma^2(BB^T)^{-1}A&-\gamma^2(BB^T)^{-1}\ea\right].
\end{align}
Thus, $\Theta_{1,3}=\Gamma^3(Q_{1,3})=-Q_{1,3}$.

{\it{2) $n>m$}}: In this case, $\Lambda_k(x,z)\neq\emptyset,~k\ge n$ for all $(x,z)\in\R^n\times\R^n$ since system \er{eq:system} is controllable by Assumption \ref{ass:system}. Set $\bar{\Phi}=\text{diag}(\Phi,\Phi,\cdots,\Phi)$,

\begin{align}
\nn
\ba{rl}
\bar{x}&\doteq\left[\ba{cccc}x_0^T&x_1^T&\cdots&x_{n-1}^T\ea\right]^T,
\\
\bar{w}&\doteq\left[\ba{cccc}w_0^T&w_1^T&\cdots&w_{n-1}^T\ea\right]^T,
\\
\bar{A}&\doteq\left[\ba{cccc}I&A^T&\cdots&(A^{n-1})^T\ea\right]^T,
\\
\bar{C}&\doteq\left[A^{n-1}B, A^{n-1}B,\cdots,AB,B\right],
\ea
&
\bar{B}\doteq\left[\ba{ccccc}0 &0&\cdots&0&0\\B &0&\cdots&0&0\\AB&B&\cdots&0&0\\ \vdots&\vdots&\vdots&\vdots&\vdots \\A^{n-2}B&A^{n-3}B&\cdots&B&0\ea\right].
\end{align}
Using this notation, the state trajectory $x_{[0,n-1]}$ generated via \er{eq:system} subject to $x_0=x, x_n=z$ can be written as
\begin{align}
\label{eq:dyna-compact}
\bar{x}=\bar{A}x+\bar{B}\bar{w},~
z=A^nx+\bar{C}\bar{w}.
\end{align}
Controllability of $(A,B)$ implies that $\text{rank}(\bar{C})=n$, i.e. $\bar{C}\bar{C}^T$ is invertible. Hence $\Lambda_n(x,z)$ of \er{eq:Lamdak} can be characterized by
$$
\Lambda_n(x,z)=\left\{\bar{w}\in\R^{mn}\,|\,z-A^nx=\bar{C}\bar{w}\right\}=\left\{\bar{C}^+(z-A^nx)+(I-\bar{C}^+\bar{C})\,\tilde{w}\,|\,\tilde{w}\in\R^{nm}\right\}.
$$
Here, $\bar{C}^+=\bar{C}^T(\bar{C}\bar{C}^T)^{-1}\in\R^{mn\times n}$ is the Moore-Penrose pseudo-inverse of $\bar{C}$. The matrix $I-\bar{C}^+\bar{C}\in\R^{mn\times{mn}}$ may not be invertible. Suppose that $\text{rank}(I-\bar{C}^+\bar{C})=r\le mn$. Then, there exists $\bar{D}\in\R^{mn\times r}$ with $\text{rank}(\bar{D})=r$ such that
$
\left\{(I-\bar{C}^+\bar{C})\,\tilde{w}\,|\,\tilde{w}\in\R^{nm}\right\}=\left\{\bar{D}\,\hat{w}\,|\,\hat{w}\in\R^r\right\}.
$
Thus, $\Lambda_n(x,z)$ can be characterized by
%\begin{align}
%\label{eq:chara-Lamda}
$$
\Lambda_n(x,z)=\left\{\bar{C}^+\,(z-A^nx)+\bar{D}\hat{w}\,|\,\hat{w}\in\R^r\right\}.
%\end{align}
$$
From \er{eq:value-constr}, \er{eq:dyna-compact},
\begin{align}
\nn
\mathrm{S}_{n,3}(x,z)&=\sup_{w\in\cW_{[0,n-1]}}\left\{\sum_{k=0}^{n-1}\left(\ts{\frac{1}{2}}\,x_k^T\Phi x_k-\ts{\frac{1}{2}}\,\gamma^2w_k^Tw_k\right)\biggl| x_n=z\right\}=\sup_{\bar{w}\in\Lambda_n(y,z)}\left\{\ts{\frac{1}{2}}\bar{x}^T\bar{\Phi}\bar{x}-\ts{\frac{1}{2}}\gamma^2\bar{w}^T\bar{w}\right\}
%\\\nn
%                           &=\sup_{\bar{w}\in\Lambda_n(y,z)}\left\{\sum_{k=0}^{n-1}\ts{\frac{1}{2}}x_k^T\Phi x_k-\ts{\frac{1}{2}}\gamma^2\sum_{k=0}^{n-1}w_k^Tw_k\right\}=\sup_{\bar{w}\in\Lambda_n(y,z)}\left\{\ts{\frac{1}{2}}\bar{x}^T\bar{\Phi}\bar{x}-\ts{\frac{1}{2}}\gamma^2\bar{w}^T\bar{x}\right\}
\\\nn
                           &=\sup_{\bar{w}\in\Lambda_n(x,z)}\left\{\ts{\frac{1}{2}}\,(\bar{A}x+\bar{B}\bar{w})^T\bar{\Phi}(\bar{A}x+\bar{B}\bar{w})-\ts{\frac{1}{2}}\,\gamma^2\bar{w}^T\bar{w}\right\}
\\\nn
                           &=\sup_{\hat{w}\in\R^r}\left\{\ba{l}\ts{\frac{1}{2}}\,(\bar{A}x+\bar{B}(\bar{C}^+(z-A^nx)+\bar{D}\hat{w}))^T\bar{\Phi}
                           (\bar{A}x+\bar{B}(\bar{C}^+(z-A^nx)+\bar{D}\hat{w}))
                           \\
                            -\ts{\frac{1}{2}}\,\gamma^2(\bar{C}^+(z-A^nx)+\bar{D}\hat{w})^T
                            (\bar{C}^+(z-A^nx)+\bar{D}\hat{w})\ea\right\}
\\\label{eq:Theta-1-5}
                           &=\frac{1}{2}\left[\ba{c}x\\z\ea\right]^TQ_{n,3}\left[\ba{c}x\\z\ea\right],\,\text{with}\, Q_{n,3}=\left[\ba{cc}\bar{R}_1^T\bar{\Phi}R_1-\gamma^2\bar{R}_3^T\bar{R}_3 & R_1^T\bar{\Phi}\bar{R}_2-\gamma^2\bar{R}_3^T\bar{R}_4
 \\
\bar{R}_2^T\bar{\Phi}\bar{R}_1-\gamma^2\bar{R}_4^T\bar{R}_3 &\bar{R}_2^T\bar{\Phi}\bar{R}_2-\gamma^2\bar{R}_4^T\bar{R}_4  \ea\right],
\end{align}
where
\be
\ba{rlrll}
\bar{R}_1&=\bar{A}-\bar{B}\bar{C}^+A^n-\bar{B}\bar{D}\bar{\Omega}^{-1}\Pi_1,~~&\bar{R}_2&=\bar{B}\bar{C}^+-\bar{B}\bar{D}\bar{\Omega}^{-1}\Pi_2,
\\\nn
\bar{R}_3&=\bar{C}^{+}A^n+\bar{D}\bar{\Omega}^{-1}\Pi_1,~~&\bar{R}_4&=-\bar{C}^++\bar{D}\bar{\Omega}^{-1}\Pi_2,
\\\nn
\Pi_1&=\bar{D}\bar{\Phi}(\bar{A}-\bar{B}\bar{C}^{+}A^n)+\gamma^2\bar{C}^+A^n,~~&\Pi_2&=\bar{D}\bar{\Phi}\bar{B}\bar{C}^{+}-\gamma^2\bar{D}\bar{C}^{+},
\\\nn
\bar{\Omega}&=\bar{D}^T(\bar{B}^T\bar{\Phi}\bar{B}-\gamma^2I)\bar{D}.
\ea
\ee
Thus, $\Theta_{n,3}=\Gamma^3(Q_{n,3})=-Q_{n,3}$.

\subsection{Computational method}
\label{sec:comp-method}

Based on Theorem \ref{thm:operator-duality}, \ref{thm:exponential} and \ref{thm:Q-Theta}, a max-plus fundamental solution based computational method can be summarized by the following steps:

\begin{center}
\vspace{3mm}
\parbox[c]{15cm}{\centering
\begin{itemize}
\item[\done]
Obtain the initial Hessian $Q_{1,i}$  using \er{eq:Q11}, \er{eq:Q12}, or \er{eq:Q-1-3}, or $Q_{n,3}$ using \er{eq:Theta-1-5}, see Section \ref{sec:ini}.

\item[\dtwo]
 Compute the matrix $\Theta_{1,i}$ via $\Theta_{1,i}=\Gamma^i(Q_{1,i})$ for $Q_{1,i}$ of \er{eq:Q11}, \er{eq:Q12}, \er{eq:Q-1-3}; Or $\Theta_{n,3}$ via $\Theta_{n,3}=\Gamma^3(Q_{n,3})$ for $Q_{n,3}$ of \er{eq:Theta-1-5}, see Theorem \ref{thm:Q-Theta}.

\item[\dthree]
 Propagate the matrices $\Theta_{k,i}, k\in\Z_{>0}$, via \er{eq:B-update}. Use $k_1=k_2=k$ for fast computation via index doubling (or $k_1=1$ and $k_2=k$ for slower linear indexing).

\item[\dfour]
Obtain the Hessian $Q_{k,i}$ for some $k\in\Z_{>0}$ via $Q_{k,i}=\Gamma^i(\Theta_{k,i})$ and \er{eq:Gamma}, see Theorem \ref{thm:Q-Theta}.

\item[\dfive]
Compute the value function $W_k$ via \er{eq:fund-solu} and \er{eq:quad-fund-primal}, together with the max-plus dual of the terminal payoff $\op{D}_{\psi^i}{\Psi}$.
\end{itemize}
\vspace{3mm}
}
\end{center}

As indicated in the above steps, this computational method predominantly involves repeated applications of the matrix operation $\circledast$ of \er{eq:circledast} in Step {\dthree}. These operations occur in the dual space, and correspond to propagation of the Hessian $\Theta_{k,i}$ of the kernel $\mathrm{B}_{k,i}$ of the max-plus integral operator $\op{B}_{k,i}$. (Recall that this operator $\op{B}_{k,i}$ defines the fundamental solution semigroup, with properties inherited from the dynamic programming evolution operator $\op{S}_{k,i}$ defined in the primal space by \er{eq:op-DPP-k}, see Remark \ref{rmk:semigroup-2}.) As this propagation $\Theta_{k,i}$ occurs in the dual space, two additional primal / dual operations are required by the computational method, see Steps {\done},{\dtwo} and {\dfour}, {\dfive}. These operations map the terminal payoff to the dual space, and the computed value function back to the primal space. Both involve maximization, see \er{eq:op-dual} and \er{eq:op-inv-dual}. However, for longer time horizons, the computational effort associated with these maximizations is dominated by the aforementioned $\Theta_{k,i}$ propagation via matrix operation \er{eq:circledast}. The computational complexity of propagating $\Theta_{1,i}$ to $\Theta_{k,i}$ in Step {\dthree} is shown to be in the order of $\log_2{k}$ in Example \ref{sec:ex-LQR}. As this operation is fast and accurate, the computational method is expected to be similarly fast and accurate, particularly on longer time horizons. This expectation is realized in the specific example considered in Section \ref{sec:ex-LQR}.

In the infinite horizon case, convergence of the iteration $\Theta_{k,i}$ is critical. This is discussed in detail in Section \ref{sec:inf}.

\section{Infinite horizon linear regulator problems}

\label{sec:inf}

The infinite horizon linear regulator problem is defined as the limit of finite horizon linear regulator problem
\er{eq:value}  as $k\rightarrow\infty$. This infinite horizon optimal control problem can be studied via convergence of the sequence of value functions $\{W_k\}_{k=0}^\infty$. Since $W_{k+1}=\op{S}_1W_k,k\in\Z_{>0}$, the convergence of $W_k\rightarrow W,k\rightarrow\infty$ implies that
$
0\otimes W=\op{S}_1W.
$
That is, the limit $W$ is a max-plus eigenvector of the operator $\op{S}_1$ corresponding to the eigenvalue $0$ (the max-plus multiplicative identity). In the special case of LQR (i.e. a linear regulator problem with a quadratic terminal payoff), this is the well-studied convergence problem of the difference Riccati equation (DRE) \er{eq:DRE} \cite{BGP85}, \cite{CM:70}. The value function of the infinite horizon LQR problem is a quadratic function characterized by the  stabilizing solution of the corresponding algebraic Riccati equation (ARE). However, for the non-quadratic linear regulator problem, the convergence of $\{W_{k}\}_{k=0}^\infty$ of \er{eq:value}
cannot be reduced to the convergence problem of DRE \er{eq:DRE}, as the value functions  $W_k, k\in\Z_{\ge0}$ are not necessarily quadratic. By employing the representation of $W_k$ of \er{eq:fund-solu}, this more general convergence problem can be investigated via the convergence of the auxiliary value functions $\{\mathrm{S}_{k,i}\}_{k=1}^\infty$ of \er{eq:kernel-primal}.

\subsection{Convergence of the fundamental solution semigroup kernels $\mathrm{B}_{k,i}$}

The sequence of quadratic functions $\{\mathrm{S}_{k,i}\}_{k=1}^\infty$ is characterized by the matrix sequence $\{Q_{k,i}\}_{k=1}^\infty$,  while the sequence of duals  $\{\mathrm{B}_{k,i}\}_{k=1}^\infty$ is characterized by the matrix sequence $\{\Theta_{k,i}\}_{k=1}^\infty$. A pair of matrices $Q_{k,i}$ and $\Theta_{k,i}$ is related by $\Gamma^i$ according to Theorem \ref{thm:Q-Theta}. Hence, the convergence of $\{\mathrm{S}_{k,i}\}_{k=1}^\infty$ and $\{\mathrm{B}_{k,i}\}_{k=1}^\infty$ is reduced to the convergence of matrix sequences $\{Q_{k,i}\}_{k=1}^\infty$ and $\{\Theta_{k,i}\}_{k=1}^\infty$ respectively.

From Theorem \ref{thm:exponential}, the sequence $\{\Theta_{k,i}\}_{k=1}^\infty$ of \er{eq:quad-fund-dual} satisfies \er{eq:B-update}, where the initial condition is given by  \er{eq:Theta-1-1}, \er{eq:Theta-1-2}, \er{eq:Q-1-3}, or \er{eq:Theta-1-5}  depending on the specific case specified there. To present a convergence result for the sequence $\{\Theta_{k,i}\}_{k=1}^n$, the convergence of a matrix sequence $\{\Omega_k\}_{k=1}^\infty$ generated by
\begin{align}
\label{eq:sequence-theta-k}
\Omega_{k+1}=\Omega_k\circledast\Omega_k,~~ \Omega_1=\Omega,
\end{align}
is proved first. Here, the initial condition $\Omega\in\R^{2n\times2n}$ takes the form
\begin{align}
\label{eq:ini-Omega}
\Omega=\left[\ba{cc} \Omega^{11} &\Omega^{12} \\ \Omega^{21} &\Omega^{22}\ea\right],
\end{align}
satisfying $(\Omega^{12})^T={\Omega^{21}}$ and $\Omega^{11}+\Omega^{22}>0$. That is, in considering \er{eq:sequence-theta-k}, convergence of the subsequence $\{\Theta_{2^k,i}\}_{k=1}^\infty$ is of interest. The following convergence result is useful in proving the convergence of this sequence.
\begin{lemma}
\label{lem:sequence}
Fix any constants $\sigma>0, \lambda>0, \rho>0$ such that
\begin{align}
\label{eq:rhosiglam}
\rho^{-2}\sigma<1, ~~ \rho\le\lambda-2\rho^{-1}\sigma(1-\rho^{-2}\sigma)^{-1}.
\end{align}
Then, the sequence $\{(\sigma_k,\lambda_k)\}_{k=1}^\infty$ defined by
\begin{align}
\label{eq:seqa}
\sigma_{k+1}=\lambda_{k}^{-2}\sigma_k^2, ~~ \lambda_{k+1}=\lambda_k-2\lambda_k^{-1}\sigma_k,~~\quad\sigma_1=\sigma,\lambda_1=\lambda
\end{align}
is convergent, with $\sigma_k\rightarrow0$, $\lambda_k>0$ for all $k\in\Z_{>0}$ and $\lambda_k\downarrow \bar{\lambda}\ge\rho$ as $k\rightarrow\infty$.
\end{lemma}
\begin{proof}
Firstly, construct a sequence $\{(\hat{\sigma}_k,\hat{\lambda}_k)\}_{k=1}^\infty$ by
\begin{align}
\label{eq:seqb}
\hat{\sigma}_{k+1}=\rho^{-2}\hat{\sigma}_k^2, ~~\hat{\lambda}_{k+1}=\hat{\lambda}_k-2\rho^{-1}\hat{\sigma}_k,~~\quad \hat{\sigma}_1=\sigma,~\hat{\lambda}_1=\lambda.
\end{align}
From the definition of $\hat{\sigma}_k$ in \er{eq:seqb}, it follows that  $\hat{\sigma}_k>0, k\in\Z_{>0}$, and
\begin{align}
\label{eq:lem-prof-1}
\sum_{k=1}^{\infty}\hat{\sigma}_k&=\sum_{k=1}^{\infty}\rho^2(\rho^{-2}\sigma)^{2^{k-1}}\le\rho^2\sum_{k=1}^{\infty}(\rho^{-2}\sigma)^{k}=\rho^2(\rho^{-2}\sigma)(1-\rho^{-2}\sigma)^{-1}=\sigma(1-\rho^{-2}\sigma)^{-1},
\end{align}
where the left-hand inequality in \er{eq:rhosiglam}  and the fact $2^{k-1}\ge k$ for $k\in \Z_{>0}$ are used.  Thus, $\hat{\sigma}_k\rightarrow0$ as $k\rightarrow\infty$. Turning to $\hat{\lambda}_k$, note that for any $k\in\Z_{>0}$, \er{eq:seqb}, \er{eq:lem-prof-1}, and the right-hand of inequality \er{eq:rhosiglam} imply that
\begin{align}
\label{eq:lem-prof-2}
\hat{\lambda}_k&=\hat{\lambda}_1-2\rho^{-1}\sum_{j=1}^{k-1}\hat{\sigma}_j>\hat{\lambda}_1-2\rho^{-1}\sum_{j=1}^{\infty}\hat{\sigma}_j\ge\lambda-2\rho^{-1}\sigma(1-\rho^{-2}\sigma)^{-1}\ge\rho>0.
\end{align}
The right-hand definition of \er{eq:seqb} also implies that  $\{\hat{\lambda}_k\}_{k=1}^\infty$ is decreasing. Hence, there exists $\hat{\lambda}\geq\rho$ such that  $\hat{\lambda}_k\downarrow\hat{\lambda}$.

Next, construct a second sequence $\{(\bar{\sigma}_k, \bar{\lambda}_k)\}_{k=1}^\infty$ by
\begin{align}
\label{eq:seqc}
\bar{\sigma}_{k}=\hat{\sigma}_k,~~                  \bar{\lambda}_{k+1}=\bar{\lambda}_k-2\bar{\lambda}_k^{-1}\bar{\sigma}_k, ~~\quad\bar{\sigma}_1=\sigma, \bar{\lambda}_1=\lambda.
\end{align}
By inspection of \er{eq:seqb} and \er{eq:seqc}, $\bar{\lambda}_1=\hat{\lambda}_1$. In order to show that $\bar{\lambda}_k\ge\hat{\lambda}_k, k\in\Z_{>0}$, using mathematical induction, suppose that this inequality holds for $k$. Then, applying \er{eq:lem-prof-2} yields
\begin{align}
\nn
\bar{\lambda}_{k+1}&=\bar{\lambda}_k-2\bar{\lambda}_k^{-1}\bar{\sigma}_k\ge\hat{\lambda}_k-2\hat{\lambda}_k^{-1}\hat{\sigma}_k>\hat{\lambda}_k-2\rho^{-1}\hat{\sigma}_k=\hat{\lambda}_{k+1}.
\end{align}
That is, $\bar\lambda_k\ge\hat\lambda_k$ implies that $\bar\lambda_{k+1}\ge\hat\lambda_{k+1}$. Similarly, induction can be applied to show that the sequence $\{(\sigma_k,\lambda_k)\}_{k=1}^\infty$ of \er{eq:seqa} satisfies
\begin{align}
\label{eq:propty-siglam}
\sigma_k\le\bar{\sigma}_k,~\lambda_k\ge\bar{\lambda}_k,~k\in\Z_{>0}.
\end{align}
By inspection of \er{eq:seqa} and \er{eq:seqc}, $\sigma_1=\bar{\sigma}_1=\sigma$ and $\lambda_1=\bar{\lambda}_1=\lambda$. Supposing that the inequality \er{eq:propty-siglam} holds for index $k$, it is required to demonstrate that \er{eq:propty-siglam} holds for index $k+1$. Applying $\bar{\lambda}_k\ge\hat{\lambda}_k\ge\rho$ and $\bar{\sigma}_k=\hat{\sigma}_k$ for $k\in\Z_{>0}$ yields
$
\sigma_{k+1}=\lambda_k^{-2}\sigma_k^{2}\le\bar{\lambda}_k^{-2}\bar{\sigma}_k^2\le\rho^{-2}\bar{\sigma}_k^2=\bar{\sigma}_{k+1}.
$
 Similarly. it can be shown that
$
\lambda_{k+1}=\lambda_k-2\lambda_k^{-1}\sigma_k\ge\bar{\lambda}_k-2\bar{\lambda}_k^{-1}\bar{\sigma}_k=\bar{\lambda}_{k+1},
$
as required.

Thus, it has been shown that
$
\sigma_k\le\bar{\sigma}_k=\hat{\sigma}_k\rightarrow0,~\lambda_k\ge\bar{\lambda}_k\ge\hat{\lambda}_k\ge\rho, k\in\Z_{>0}.
$
By inspection of the definition ${\sigma_k}$ in \er{eq:seqa}, $\sigma_k>0, k\in\Z_{>0}$. Thus $\sigma_k\rightarrow0, k\rightarrow\infty$.  It follows immediately from \er{eq:seqa} that the sequence $\{\lambda_k\}_{k=1}^\infty$ is decreasing. Thus, there exists $\bar{\lambda}\ge\rho$ such that ${\lambda}_k\downarrow\bar{\lambda}$.
\end{proof}

By applying Lemma \ref{lem:sequence},  next theorem proves convergence of the sequence $\{\Omega_k\}_{k=1}^\infty$ specified by \er{eq:sequence-theta-k}.
\begin{theorem}
\label{thm:conv-circled}
Fix any constants $\sigma>0,\lambda>0,\rho>0$ such that \er{eq:rhosiglam} holds. Suppose that the matrix  $\Omega$ of \er{eq:ini-Omega} satisfies
\begin{align}
\label{eq:ineq-Omega}
\Omega^{12}\Omega^{21}\le\sigma I,~~ \Omega^{21}\Omega^{12}\le\sigma I, ~~ \Omega^{11}+\Omega^{22}\ge\lambda I.
\end{align}
Then, the matrix sequence $\{\Omega_k\}_{k=1}^\infty$ specified by \er{eq:sequence-theta-k} satisfies $\Omega_k^{11}+\Omega_k^{22}\ge\rho I, k\in\Z_{>0}$, and there exists a matrix $\Omega_{\infty}=\text{diag}(\Omega_{\infty}^{11},\Omega_{\infty}^{22})$
such that $\Omega_{\infty}^{11}+\Omega_{\infty}^{22}\ge \rho I$ and
$
\Omega_k\rightarrow\Omega_{\infty},~k\rightarrow\infty.
$
\end{theorem}
\begin{proof}
By definition of \er{eq:circledast} $\circledast$ operation ,
\be
\label{eq:Omega-update}
\ba{rlll}
\Omega_{k+1}^{11}&=\Omega_{k}^{11}-\Omega_k^{12}(\Omega_{k}^{11}+\Omega_{k}^{22})^{-1}\Omega_k^{21},&\Omega_{k+1}^{12}&=-\Omega_k^{12}(\Omega_k^{11}+\Omega_k^{22})^{-1}\Omega_k^{12},
\\
\Omega_{k+1}^{21}&=-\Omega_k^{21}(\Omega_k^{11}+\Omega_k^{22})^{-1}\Omega_k^{21},&\Omega_{k+1}^{22}&=\Omega_{k}^{22}-\Omega_k^{21}(\Omega_{k}^{11}+\Omega_{k}^{22})^{-1}\Omega_k^{12}.
\ea
\ee
It will be shown by mathematical induction that for any $k\in\Z_{>0}$,
\begin{align}
\label{eq:ineq-Omega-k}
\Omega_k^{12}\Omega_k^{21}\le\sigma_k I,~~ \Omega_k^{21}\Omega_k^{12}\le\sigma_k I, ~~ \Omega_k^{11}+\Omega_k^{22}\ge\lambda_k I,
\end{align}
where  $\{(\sigma_k,\lambda_k)_{k=1}^\infty$ are as per \er{eq:seqa}. The $k=1$ case is immediate from \er{eq:sequence-theta-k}, \er{eq:ini-Omega}, and \er{eq:ineq-Omega}. Suppose that \er{eq:ineq-Omega-k} holds for $k$, \er{eq:ineq-Omega-k} is required to hold for $k+1$. From \er{eq:Omega-update} and \er{eq:ineq-Omega-k},
\begin{align}
\nn
\Omega_{k+1}^{12}\Omega_{k+1}^{21}&=\Omega_k^{12}(\Omega_k^{11}+\Omega_k^{22})^{-1}\Omega_k^{12}\Omega_k^{21}(\Omega_k^{11}+\Omega_k^{22})^{-1}\Omega_k^{21}\le\lambda_k^{-2}\sigma_k^2I=\sigma_{k+1}I.
\end{align}
A similar argument proves that
$
\Omega_{k+1}^{21}\Omega_{k+1}^{12}\le\sigma_{k+1}I.
$
From \er{eq:Omega-update},
\begin{align}
\nn
\Omega_{k+1}^{11}+\Omega_{k+1}^{22}&=\Omega_{k}^{11}+\Omega_{k}^{22}-\Omega_k^{12}(\Omega_{k}^{11}+\Omega_{k}^{22})^{-1}\Omega_k^{21}-\Omega_k^{21}(\Omega_{k}^{11}+\Omega_{k}^{22})^{-1}\Omega_k^{12}
\\\label{eq:omega-diag}
&\ge\lambda_kI-2\lambda_k^{-1}\sigma_kI=\lambda_{k+1}I\ge \rho I>0.
\end{align}

According to Lemma \ref{lem:sequence}, $\sigma_k\downarrow0, k\rightarrow\infty$, and there exists $\bar{\lambda}>0$ such that $\lambda_k\downarrow\bar{\lambda}>0, k\rightarrow\infty$, where \er{eq:rhosiglam} is assumed as per the Theorem statement. Since $\Omega_k^{12}=(\Omega_k^{21})^T, k\in\Z_{>0}$, \er{eq:ineq-Omega-k} implies that $
||\Omega_k^{12}||_2\le \sqrt{\sigma_k}\downarrow0,~||\Omega_k^{21}||_2\le \sqrt{\sigma_k}\downarrow0,~k\in\Z_{>0},
$
where $||\cdot||_2$ denotes the matrix spectra norm. Thus,
$
\Omega_k^{12}\rightarrow0,~\Omega_k^{21}\rightarrow0,~k\rightarrow\infty.
$
From (\ref{eq:Omega-update}) and \er{eq:omega-diag},
\begin{align}
\label{eq:prof-extra-3}
||\Omega_{k-1}^{11}-\Omega_k^{11}||_2&=||\Omega_k^{12}(\Omega_{k}^{11}+\Omega_{k}^{22})^{-1}\Omega_k^{21}||_2\le||\Omega_k^{12}||_2||\Omega_k^{21}||_2||(\Omega_{k}^{11}+\Omega_{k}^{22})^{-1}||_2\le\sigma_k\rho^{-1}.
\end{align}
From Lemma 4.1, $\lambda_k\ge\rho>0$, $\sigma_k\le\sigma$, and
$$
\sigma_{k+1}=\lambda_k^{-2}\sigma_k^2=(\lambda_k^{-2}\sigma_k)\sigma_k\le(\rho^{-2}\sigma)\sigma_k.
$$
Hence \er{eq:prof-extra-3} turns into
\begin{align}
\label{eq:prof-extra-1}
\|\Omega_{k-1}^{11}-\Omega_k^{11}\|_2\le (\rho^{-2}\sigma)\sigma_{k-1}\rho^{-1}.
\end{align}
Note that it is assumed that $\rho^{-2}\sigma<1$. Fix any $p,q\in\Z_{>0}$ such that $p<q$. Applying \er{eq:prof-extra-1}
\begin{align}
\nn
 \|\Omega_p^{11}-\Omega_q^{11}\|_2&\le \|\Omega_p^{11}-\Omega_{p+1}^{11}\|_2+ \|\Omega_{p+1}^{11}-\Omega_{p+2}^{11}\|_2+\cdots+ \|\Omega_{q-1}^{11}-\Omega_{q}^{11}\|_2
                                                         \\\nn
                                                         &\le (\rho^{-2}\sigma)\sigma_p\rho+(\rho^{-2}\sigma)^2\sigma_p\rho+\cdots+(\rho^{-2}\sigma)^{q-p}\sigma_p\rho
                                                         \\\nn
                                                         &=\frac{\rho^{-2}\sigma-(\rho^{-2}\sigma)^{q-p+1}}{1-\rho^{-2}\sigma} \rho \sigma_p
\\\nn
&\le \frac{\rho^{-1}\sigma}{1-\rho^{-2}\sigma}\sigma_p.
\end{align}
Thus $\|\Omega_p^{11}-\Omega_q^{11}\|_{2}\rightarrow 0$ as $p\rightarrow\infty$ since $\sigma_p\rightarrow0$. Hence, the sequence $\{\Omega_k^{11} \}_{k=1}^\infty$ is a Cauchy sequence. Consequently, there exists $\Omega_{\infty}^{11}$ such that $\Omega_k^{11}\rightarrow\Omega_{\infty}^{11}, k\rightarrow\infty$. It can be similarly shown that there exists $\Omega_{\infty}^{22}$ such that $\Omega_k^{22}\rightarrow\Omega_{\infty}^{22}, k\rightarrow\infty$. From \er{eq:omega-diag}, $\Omega_{\infty}^{11}+\Omega_{\infty}^{22}\ge \rho I$.
\end{proof}

Applying Theorem \ref{thm:conv-circled} to the matrices $\Theta_{1,i}$ of \er{eq:Theta-1-1}, \er{eq:Theta-1-2}, or \er{eq:Q-1-3} leads to convergence of a subsequence $\{\Theta_{2^k,i}\}_{k=1}^\infty$. Applying Theorem \ref{thm:conv-circled} to the matrices $\Theta_{n,3}$ of \er{eq:Theta-1-5} leads to convergence of a subsequence $\{\Theta_{n2^k,3}\}_{k=1}^\infty$.   To prove the convergence of the sequence $\{\Theta_{k,i}\}_{k=1}^\infty$, the following result is useful.

\begin{theorem}
\label{thm:mono-Theta}
Fix any $i\in\{1,2,3\}$ and constants  $\sigma>0,\lambda>0,\rho>0$ such that \er{eq:rhosiglam} holds. Suppose that inequality \er{eq:ineq-Omega} holds for a matrix $\Theta_{p,i}, p\in\Z_{>0}$ in the sequence $\{\Theta_{k,i}\}_{k=1}^\infty$ of \er{eq:quad-fund-dual}.  Then, the subsequence $\{\Theta_{kp,i}\}_{k=1}^\infty$ satisfies
\begin{align}
\label{eq:mono-Theta}
\left\{\ba{rlrl}\Theta_{(k+1)p,i}^{11}&\le\Theta_{kp,i}^{11},& \Theta_{(k+1)p,i}^{12}\Theta_{(k+1)p,i}^{21}&\le\Theta_{kp,i}^{12}\Theta_{kp,i}^{21},
\\
\Theta_{(k+1)p,i}^{21}\Theta_{(k+1)p,i}^{12}&\le\Theta_{kp,i}^{21}\Theta_{kp,i}^{12},&\Theta_{(k+1)p,i}^{22}&\le\Theta_{kp,i}^{22}.\ea\right.
\end{align}
\end{theorem}
\begin{proof}
From Theorem \ref{thm:exponential}, the sequence $\{\Theta_{kp,i}\}_{k=1}^\infty$ satisfies $\Theta_{(k+1)p,i}=\Theta_{p,i}\circledast\Theta_{kp,i}, k\in\Z_{>0}$. From definition \er{eq:circledast} of $\circledast$, $\Theta_{(k+1)p,i}=\Theta_{p,i}\circledast\Theta_{kp,i}=\Theta_{kp,i}\circledast\Theta_{p,i}$. That is,
\begin{align}
\label{eq:thm-mono-prof-1}
\left[\ba{cc}\Theta_{(k+1)p,i}^{11}&\Theta_{(k+1)p,i}^{12}\\
\Theta_{(k+1)p,i}^{21}&\Theta_{(k+1)p,i}^{22}\ea\right]&=\left[\ba{cc}  \Theta_{p,i}^{11}-\Theta_{p,i}^{12}(\Theta_{p,i}^{22}+\Theta_{kp,i}^{11})^{-1}\Theta_{p,i}^{21}&-\Theta_{p,i}^{12}(\Theta_{p,i}^{22}+\Theta_{kp,i}^{11})^{-1}\Theta_{kp,i}^{12}\\-\Theta_{kp,i}^{21}(\Theta_{p,i}^{22}+\Theta_{kp,i}^{11})^{-1}\Theta_{p,i}^{21}&\Theta_{kp,i}^{22}-\Theta_{kp,i}^{21}(\Theta_{p,i}^{22}+\Theta_{kp,i}^{11})^{-1}\Theta_{kp,i}^{12}\ea\right]
\\\nn
            &=\left[\ba{cc}  \Theta_{kp,i}^{11}-\Theta_{kp,i}^{12}(\Theta_{kp,i}^{22}+\Theta_{p,i}^{11})^{-1}\Theta_{kp,i}^{21}&-\Theta_{kp,i}^{12}(\Theta_{kp,i}^{22}+\Theta_{p,i}^{11})^{-1}\Theta_{p,i}^{12}\\-\Theta_{p,i}^{21}(\Theta_{kp,i}^{22}+\Theta_{p,i}^{11})^{-1}\Theta_{kp,i}^{21}&\Theta_{p,i}^{22}-\Theta_{p,i}^{21}(\Theta_{kp,i}^{22}+\Theta_{p,i}^{11})^{-1}\Theta_{p,i}^{12}\ea\right].
\end{align}
With a view to applying an inductive argument to prove the $\Theta_{kp,i}^{11}$ and $\Theta_{kp,i}^{22}$ inequalities in \er{eq:mono-Theta}, note first that in the $k=1$ case,  \er{eq:thm-mono-prof-1} implies that
\begin{align}
\nn
\Theta_{2p,i}^{11}&=\Theta_{p,i}^{11}-\Theta_{p,i}^{12}(\Theta_{p,i}^{11}+\Theta_{p,i}^{22})^{-1}\Theta_{p,i}^{21}\le\Theta_{p,i}^{11},
\\\nn
\Theta_{2p,i}^{22}&=\Theta_{p,i}^{22}-\Theta_{p,i}^{21}(\Theta_{p,i}^{11}+\Theta_{p,i}^{22})^{-1}\Theta_{p,i}^{12}\le\Theta_{p,i}^{22},
\end{align}
where the assumption that $\Theta_{p,i}^{11}+\Theta_{p,i}^{22}\ge\lambda I>0$ is used. Assume that for any $k>1$,
\begin{align}
\label{eq:thm-mono-prof-2}
\Theta_{kp,i}^{11}\le\Theta_{(k-1)p,i}^{11}\le\cdots\le\Theta_{p,i}^{11}, \quad \Theta_{kp,i}^{22}\le\Theta_{(k-1)p,i}^{22}\le\cdots\le\Theta_{p,i}^{22}.
\end{align}
Then, from \er{eq:thm-mono-prof-1},
\begin{align}
\nn
\Theta_{kp,i}^{11}-\Theta_{(k+1)p,i}^{11}&=\Theta_{kp,i}^{12}(\Theta_{kp,i}^{22}+\Theta_{p,i}^{11})^{-1}\Theta_{kp,i}^{21}\ge\Theta_{kp,i}^{12}(\Theta_{p,i}^{22}+\Theta_{p,i}^{11})^{-1}\Theta_{kp,i}^{21}\ge0,
\\\label{eq:thm-mono-prof-3}
\Theta_{kp,i}^{22}-\Theta_{(k+1)p,i}^{22}&=\Theta_{kp,i}^{21}(\Theta_{p,i}^{22}+\Theta_{kp,i}^{11})^{-1}\Theta_{kp,i}^{12}\ge\Theta_{kp,i}^{21}(\Theta_{p,i}^{22}+\Theta_{p,i}^{11})^{-1}\Theta_{kp,i}^{12}\ge0.
\end{align}
This proves the inequalities for $\Theta_{kp,i}^{11}$ and $\Theta_{kp,i}^{22}$ of \er{eq:mono-Theta}. From Theorem \ref{thm:conv-circled}, $\Theta_{2^k p,i}^{11}+\Theta_{2^k p,i}^{22}\ge\rho\,I$ for any $k\in\Z_{>0}$. The proved inequalities of $\Theta_{(k+1)p,i}^{11}\le\Theta_{kp,i}^{11}$ and $\Theta_{(k+1)p,i}^{22}\le\Theta_{kp,i}^{22}$ for any $k\in\Z_{>0}$ in \er{eq:mono-Theta} imply that for $q\in\Z_{>0}$
\begin{align}
\label{eq:prof-extra}
\Theta_{(k+q)p,i}^{11}\le\Theta_{kp,i}^{11},\quad \Theta_{(k+q)p,i}^{22}\le\Theta_{kp,i}^{22}.
\end{align}
For any $k\in\Z_{>0}$, it holds $2^k>k$. Thus, $\hat q(k)\doteq 2^k-k\in\Z_{>0} $. Applying $\hat q(k)$ in \er{eq:prof-extra} yields
\begin{align}
\nn
\Theta_{kp,i}^{11}+\Theta_{kp,i}^{22}&\ge \Theta_{(k+\hat q(k))p,i}^{11}+\Theta_{(k+\hat q(k))p,i}^{22}
\\\label{eq:thm-mono-prof-4}
                                                           &=\Theta_{2^kp,i}^{11}+\Theta_{2^kp,i}^{22}
                                                          \\\nn
                                                          &\ge\rho\,I.
\nn
\end{align}
To show the inequalities of $\Theta_{kp,i}^{12}\Theta_{kp,i}^{21}$ and $\Theta_{kp,i}^{21}\Theta_{kp,i}^{12}$ in \er{eq:mono-Theta}, using inequality \er{eq:rhosiglam}, \er{eq:ineq-Omega}, \er{eq:thm-mono-prof-1}, \er{eq:thm-mono-prof-2} and \er{eq:thm-mono-prof-4},
\begin{align}
\nn
\Theta_{(k+1)p,i}^{12}\Theta_{(k+1)p,i}^{21}&=\Theta_{kp,i}^{12}(\Theta_{kp,i}^{22}+\Theta_{p,i}^{11})^{-1}\Theta_{p,i}^{12}\Theta_{p,i}^{21}(\Theta_{kp,i}^{22}+\Theta_{p,i}^{11})^{-1}\Theta_{kp,i}^{21}
\\\nn
                                            &\le\sigma\Theta_{kp,i}^{12}(\Theta_{kp,i}^{22}+\Theta_{kp,i}^{11})^{-1}(\Theta_{kp,i}^{22}+\Theta_{kp,i}^{11})^{-1}\Theta_{kp,i}^{21}
                                            \\\nn
                                            &\le \sigma\rho^{-2}\Theta_{kp,i}^{12}\Theta_{kp,i}^{21}
                                            \\\nn
                                            &<\Theta_{kp,i}^{12}\Theta_{kp,i}^{21}.
\end{align}
A similar argument shows that $\Theta_{(k+1)p,i}^{21}\Theta_{(k+1)p,i}^{12}\le\Theta_{kp,i}^{21}\Theta_{kp,i}^{12}$.
\end{proof}

Combining Theorem \ref{thm:conv-circled} and Theorem \ref{thm:mono-Theta}, the convergence of the sequence $\{\Theta_{k,i}\}_{k=1,n}^\infty$ that characterizes the kernels $\mathrm{B}_{k,i}, k\in\Z_{>0}$, of \er{eq:dual-fund-1} can be proved. Two cases are considered separately. The first is for the sequence $\{\Theta_{k,i}\}_{k=1}^\infty$ initialized with $\Theta_{1,i}$ from \er{eq:Theta-1-1}, \er{eq:Theta-1-2}, or \er{eq:Q-1-3}, while the second one is for the sequence $\{\Theta_{k,3}\}_{k=n}^\infty$, initialized with $\Theta_{n,3}$ from \er{eq:Theta-1-5}.

\begin{theorem}
\label{thm:conv-Theta_k-1}
Fix any $i\in\{1,2,3\}$ and constants  $\sigma>0,\lambda>0,\rho>0$ such that \er{eq:rhosiglam} holds.

\begin{enumerate}

\item
\label{thm:conv-Themta-item1}
Suppose that inequality \er{eq:ineq-Omega} holds for the matrices $\Theta_{1,i}, i\in\{1,2,3\}$ of \er{eq:Theta-1-1}, \er{eq:Theta-1-2}, or \er{eq:Q-1-3}. Then, the matrix sequence $\{\Theta_{k,i}\}_{k=1}^\infty$ of \er{eq:quad-fund-dual} converges to a block diagonal matrix $\Theta_{\infty,i}=\text{diag}(\Theta_{\infty,i}^{11},\Theta_{\infty,i}^{22})$ such that $\Theta_{\infty,i}^{11}+\Theta_{\infty,i}^{22}\ge\rho\,I$.

\item
\label{thm:conv-Themta-item2}
Suppose that inequality \er{eq:ineq-Omega} holds for the matrix $\Theta_{n,3}$ of \er{eq:Theta-1-5}. Then, the matrix sequence $\{\Theta_{k,3}\}_{k=n}^\infty$ of \er{eq:quad-fund-dual} initialized with $\Theta_{n,3}$ of \er{eq:Theta-1-5} converges to a block diagonal matrix $\Theta_{\infty,i}=\text{diag}(\Theta_{\infty,i}^{11},\Theta_{\infty,i}^{22})$ such that $\Theta_{\infty,i}^{11}+\Theta_{\infty,i}^{22}\ge\rho\,I$.
\end{enumerate}
\end{theorem}
\begin{proof}
{\it \ref{thm:conv-Themta-item1}):} From Theorem \ref{thm:conv-circled}, the subsequence $\{\Theta_{2^k,i}\}_{k=1}^\infty $ initialized from the matrices $\Theta_{1,i}, i\in\{1,2,3\}$ of \er{eq:Theta-1-1}, \er{eq:Theta-1-2}, or \er{eq:Q-1-3}, converges to a block diagonal matrix $\Theta_{\infty, i}$ as $k\rightarrow\infty$. Thus, $\Theta_{2^k,i}^{12}\Theta_{2^k,i}^{21}\rightarrow0$ and $\Theta_{2^k,i}^{21}\Theta_{2^k,i}^{12}\rightarrow0$, and $\Theta_{2^k,i}^{11}+\Theta_{2^k,i}^{22}\rightarrow\Theta_{\infty,i}^{11}+\Theta_{\infty,i}^{22}\ge\rho\,I$ as $k\rightarrow\infty$. Applying the inequality \er{eq:mono-Theta} in Theorem \ref{thm:mono-Theta} for $p=1$ leads to
$\Theta_{k+1,i}^{11}\le\Theta_{k,i}^{11},\Theta_{k+1,i}^{22}\le\Theta_{k,i}^{22}$, $\Theta_{k+1,i}^{12}\Theta_{k+1,i}^{21}\le\Theta_{k,i}^{12}\Theta_{k,i}^{21}$  and $\Theta_{k+1,i}^{21}\Theta_{k+1,i}^{12}\le\Theta_{k,i}^{21}\Theta_{k,i}^{12}$ for all $k\in\Z_{>0}$. Thus,  $\Theta_{k,i}^{12}\Theta_{k,i}^{21}\rightarrow0$ and $\Theta_{k,i}^{21}\Theta_{k,i}^{12}\rightarrow0$, and $\Theta_{k,i}^{11}+\Theta_{k,i}^{22}\rightarrow\Theta_{\infty,i}^{11}+\Theta_{\infty,i}^{22}$ as $k\rightarrow\infty$.

\vspace{0.5cm}

{\it \ref{thm:conv-Themta-item2}):} Applying Theorem \ref{thm:mono-Theta} for $p=n$ and adopting a similar argument as in the proof of {\it  \ref{thm:conv-Themta-item1})} above proves that the subsequence $\Theta_{kn,3}\rightarrow\Theta_{\infty,3}$ as $k\rightarrow\infty$, with $\Theta_{\infty,3}=\text{diag}(\Theta_{\infty,3}^{11},\Theta_{\infty,3}^{22})$. According to \er{eq:Gamma}, the sequence $\{Q_{k,3}\}_{k=n}^{\infty}$ of \er{eq:quad-fund-primal} is related to $\{\Theta_{k,3}\}_{k=n}^\infty$ by $Q_{k,3}=-\Theta_{k,3}, k\ge n, k\in\Z_{>0}$. Thus, $Q_{kn,3}\rightarrow Q_{\infty,3}\doteq-\Theta_{\infty,3}$ as $k\rightarrow\infty$. The subsequence $\{Q_{kn+1,3}\}_{k=1}^\infty$ can be obtained  by applying iterations \er{eq:dynamics-Q} of Theorem \ref{thm:dynamics-Q}, with $Q_{kn,3}$ replacing $Q_{k,i}$ in the right-hand side, that is,
\begin{align}
\nn
Q_{kn+1,3}^{11}&=\Phi+A^TQ_{kn,3}^{11}A+A^TQ_{kn,3}^{11}B(\gamma^2I-B^TQ_{kn,3}^{11}B)^{-1}B^TQ_{kn,3}^{11}A,
\\\nn
Q_{kn+1,3}^{12}&=A^TQ_{kn,3}^{12}+A^TQ_{k,3}^{11}B(\gamma^2I-B^TQ_{kn,3}^{11}B)^{-1}B^TQ_{kn,3}^{12},
\\\label{eq:dynamics-Q-1}
Q_{kn+1,3}^{21}&=(Q_{kn+1,3}^{12})^T,
\\\nn
Q_{kn+1,3}^{22}&=Q_{kn,3}^{22}+Q_{kn,3}^{21}B(\gamma^2I-B^TQ_{kn,3}^{11}B)^{-1}B^TQ_{kn,3}^{12}.
\end{align}
Suppose that $Q_{kn+1,3}\rightarrow \widehat{Q}, k\rightarrow \infty$. Sending $k\rightarrow\infty$ in both sides of \er{eq:dynamics-Q-1} yields
\begin{align}
\nn
\widehat{Q}^{11}&=\Phi+A^TQ_{\infty,3}^{11}A+A^TQ_{\infty,3}^{11}B(\gamma^2I-B^TQ_{\infty,3}^{11}B)^{-1}B^TQ_{\infty,3}^{11}A,
\\\nn
\widehat{Q}^{12}&=A^TQ_{\infty,3}^{12}+A^TQ_{\infty,3}^{11}B(\gamma^2I-B^TQ_{\infty,3}^{11}B)^{-1}B^TQ_{\infty,3}^{12},
\\\label{eq:dynamics-limit}
\widehat{Q}^{21}&=(\widehat{Q}^{12})^T,
\\\nn
\widehat{Q}^{22}&=Q_{\infty,3}^{22}+Q_{\infty,3}^{21}B(\gamma^2I-B^TQ_{\infty,3}^{11}B)^{-1}B^TQ_{\infty,3}^{12}.
\end{align}
Since $Q_{\infty,3}^{12}=Q_{\infty,3}^{21}=0$, it is immediate from the second and third equation of \er{eq:dynamics-limit} that $\widehat{Q}^{12}=\widehat{Q}^{21}=0$. In a recent paper \cite{ZD4:13}, it has been established that $Q_{\infty,3}^{11}$ is the stabilising solution (minimum solution) of the Algebraic Riccati Equation (ARE)
\begin{align} 
\nn
P=\Phi+A^TPA+A^TPB(\gamma^2I-B^TPB)^{-1}B^TPA.
\end{align}
That is, $\gamma^2\,I-B^TQ_{\infty,3}^{11}B>0$ and
\begin{align}
\nn
Q_{\infty,3}^{11}=\Phi+A^TQ_{\infty,3}^{11}A+A^TQ_{\infty,3}^{11}B(\gamma^2I-B^TQ_{\infty,3}^{11}B)^{-1}B^TQ_{\infty,3}^{11}A.
\end{align}  
Thus, the first and fourth equations of \er{eq:dynamics-limit} imply that $\widehat{Q}^{11}=Q_{\infty,3}^{11}$ and $\widehat{Q}^{22}=Q_{\infty,3}^{22}$. This shows that $\widehat{Q}= Q_{\infty,3}$. Hence, the convergence of $Q_{kn,3}\rightarrow Q_{\infty,3}, k\rightarrow\infty$ implies that  $Q_{kn+1,3}\rightarrow Q_{\infty,3}$ as $k\rightarrow \infty$. In a similar way, subsequences $\{Q_{kn+j+1,3}\}_{k=1}^\infty, j=1,2,\cdots,n-2$ can be generated from $\{Q_{kn+j,3}\}_{k=1}^\infty, j=1,2,\cdots,n-2$ by using \er{eq:dynamics-Q} iteratively with respect to $j$. These $n-1$ subsequences each converge to $Q_{\infty,3}$. Consequently, the corresponding $n$ subsequences $\{\Theta_{kn+j,3}\}_{k=1}^\infty$ converge to $\Theta_{\infty,3}, j=0,1,2,\cdots,n-1$ as $k\rightarrow\infty$. Define $\iota_k\doteq\max_{j\in\{0,1,\cdots,n-1\}}\{\|(\Theta_{kn+j,3}^{11}+\Theta_{kn+j,3}^{22})-(\Theta_{\infty,3}^{11}+\Theta_{\infty,3}^{22})\|_2\}$, $\eta_k^1\doteq\max_{j\in\{0,1,\cdots,n-1\}}\{||\Theta_{kn+j,3}^{12}||_2\}$, and $\eta_k^2\doteq\max_{j\in\{0,1,\cdots,n-1\}}\{||\Theta_{kn+j,3}^{21}||_2\}$, all for $k\in\Z_{>0}$.
The convergence of subsequences $\{\Theta_{kn+j,3}\}_{k=1}^\infty$ to $\Theta_{\infty,3}$ for all $j\in\{0,1,2,\cdots,n-1\}$ implies that $\eta_k^1\rightarrow0, \eta_k^2\rightarrow0,\iota_{k}\rightarrow0$ as $k\rightarrow\infty$. For any $k\in\Z_{>0}$, let $\chi(k)\doteq\lfloor \,\frac{k}{n}\,\rfloor$.
Thus, for the sequence $\{\Theta_{k,3}\}_{k=n}^\infty$,
$||\Theta_{k,3}^{12}||_2\le\eta_{\lfloor \frac{k}{n}\rfloor}^1\rightarrow0$. Similarly, $||\Theta_{k,3}^{21}||_2\le\eta_{\lfloor \frac{k}{n}\rfloor}^2\rightarrow0$ and $||(\Theta_{k,3}^{11}+\Theta_{k,3}^{22})-(\Theta_{\infty,3}^{11}+\Theta_{\infty,3}^{22})||_2\le\iota_{\lfloor \frac{k}{n}\rfloor}\rightarrow0$. This proves the convergence of the sequence $\{\Theta_{k,3}\}_{k=n}^\infty$.
\end{proof}

\subsection{Convergence of the infinite horizon linear regulator problem}

When $\Theta_{k,i}\rightarrow\Theta_{\infty,i}, k\rightarrow\infty$, with the limit being block diagonal $\Theta_{\infty,i}=\text{diag}(\Theta_{\infty,i}^{11},\Theta_{\infty,i}^{22})$, $\Theta_{\infty,i}^{11}+\Theta_{\infty,i}^{22}>0$, the matrices $Q_{k,i}=\Gamma^i(\Theta_{k,i})\rightarrow{Q}_{\infty,i}=\Gamma^i({\Theta}_{\infty,i})$. From the definition of $\Gamma^i$ in \er{eq:Gamma}, $Q_{\infty,i}=\Gamma^i({\Theta}_{\infty,i})$ takes the form
\begin{align}
\nn
Q_{\infty,1}&=\text{diag}\left((\Theta_{\infty,1}^{11})^{-1},-\Theta_{\infty,1}^{22}\right),
\\\nn
%Q_{\infty,2}&=\text{diag}(M(\Theta_{\infty,2}^{11}+M)^{-1}M-M,M(\Theta_{\infty,2}^{11}+M)^{-1}M-\Theta_{\infty,2}^{22} ),
%\\\nn
Q_{\infty,2}&=\text{diag}(M(\Theta_{\infty,2}^{11}+M)^{-1}M-M,-\Theta_{\infty,2}^{22} ),
\\\nn
Q_{\infty,3}&=\text{diag}(-\Theta_{\infty,3}^{11},-\Theta_{\infty,3}^{22}).
\end{align}
The limit of $\mathrm{S}_{k,i}$ in \er{eq:kernel-primal} takes the form
\begin{align}
\nn
\mathrm{S}_{\infty,i}(x,z)=\frac{1}{2}\left[\ba{c}x\\z\ea\right]^{T}Q_{\infty,i} \left[\ba{c}x\\z\ea\right]=\frac{1}{2}(x^TQ_{\infty,i}^{11} x+z^TQ_{\infty,i}^{22} z).
\end{align}
Using the convergence of $\{\mathrm{S}_{k,i}\}_{k=1}^{\infty}$, a convergence result for the sequence of value functions $\{W_k\}_{k=0}^\infty$ of \er{eq:value} can be obtained by employing the representation \er{eq:fund-solu}.

\begin{theorem}
\label{thm:limit}
Suppose that (i) the sequence $\{Q_{k,i}\}_{k=1}^\infty$ defining the functions $\{\mathrm{S}_{k,i}\}_{k=1}^\infty$ of \er{eq:kernel-primal} satisfies $Q_{k,i}\rightarrow Q_{\infty,i}, k\rightarrow\infty$ with $Q_{\infty,i}=\text{diag}(Q_{\infty,i}^{11},Q_{\infty,i}^{22})$, (ii) the dual of the terminal payoff $\widehat{\Psi}^i(z)\doteq(\op{D}_{\psi^i}\Psi)(z),z\in\R^n$, is continuous, and (iii) there exist $r_0>0,\eps_0>0$ such that
\begin{align}
\label{ineq:thm-Psi}
\widehat{\Psi}^i(z)\le -\textstyle{\frac{1}{2}}z^T (Q_{\infty,i}^{22}+\eps_0 I)z,~\forall~|z|>r_0.
\end{align}
Then, $W_k(x)\rightarrow W_\infty(x),x\in\R^n$, where $W_\infty(x)$ is given by
\begin{align}
\label{eq:infty-value}
W_\infty(x)\doteq\textstyle{\frac{1}{2}}x^TQ_{\infty,i}^{11}x+\kappa,~~\text{with}~
\kappa\doteq\int_{\R^n}^\oplus\widehat{\Psi}^i(z)\otimes \left(\textstyle{\frac{1}{2}}z^TQ_{\infty,i}^{22}z\right)\,dz.
\end{align}
\end{theorem}
\begin{proof}
Fix any $x\in\R^n$. From \er{eq:fund-solu},
\begin{align}
\nn
W_k(x)&=\int_{\R^n}^\oplus \widehat{\Psi}^i(z)\otimes\left({\frac{1}{2}}\left[\ba{c}x\\z\ea\right]^{T}Q_{k,i} \left[\ba{c}x\\z\ea\right]\right)\,dz
\\\nn
      &=\textstyle{\frac{1}{2}}x^TQ_{k,i}^{11}x+\int_{\R^n}^\oplus \widehat{\Psi}^i(z)\otimes\left(\textstyle{\frac{1}{2}}z^TQ_{k,i}^{22}z\right)\otimes \left(x^TQ_{k,i}^{12}z\right)\,dz
\\\nn
      &=\textstyle{\frac{1}{2}}x^TQ_{k,i}^{11}x+\ds{\sup_{z\in\R^n}}\{f_{k,i}^x(z)\}
\end{align}
where $f_{k,i}^x:\R^n\rightarrow\R$ is
$
f_{k,i}^x(z)\doteq\widehat{\Psi}^i(z)+\textstyle{\frac{1}{2}}z^TQ_{k,i}^{22}z+ x^TQ_{k,i}^{12}z.
$
By assumption (i), $Q_{k,i}\rightarrow Q_{\infty,i}=\text{diag}(Q_{\infty,i}^{11},Q_{\infty,i}^{22})$, Theorem \ref{thm:limit} is proved if it is shown that
\begin{align}
\label{prof:thm-limit-1}
\lim_{k\rightarrow\infty}\ds{\sup_{z\in\R^n}}\left\{f_{k,i}^x(z)\right\}=\kappa.
\end{align}
To prove \er{prof:thm-limit-1}, it is first shown that there exists $K\in\Z_{>0}, \bar{r}\in\R_{>0}$ such that
\begin{align}
\label{prof:thm-limit-2}
\sup_{z\in\R^n}\{f_{k,i}^x(z)\}=\max_{|z|\le \bar{r}}\{f_{k,i}^x(z)\}, \quad \forall ~k\ge K.
\end{align}
Since $Q_{k,i}^{22}\rightarrow Q_{\infty,i}^{22}$ and $Q_{k,i}^{12}\rightarrow0$ by assumption (i), there exists $K\in\Z_{>0}$ and $r_1\ge r_0$ such that
\begin{align}
\label{prof:thm-limit-3}
Q_{k,i}^{22}-Q_{\infty,i}^{22}\le\frac{1}{2}\eps_0\,I, \quad |Q_{k,i}^{12}x|\le \frac{1}{8}\eps_0 r_1
\end{align}
for all $k\ge K$. Then, for any $r\ge  r_1$,
\begin{align}
\nn
\sup_{|z|>r}\{f_{k,i}^x(z)\}&=\sup_{|z|>r}\{\widehat{\Psi}^i(z)+\textstyle{\frac{1}{2}}z^TQ_{k,i}^{22}z+ x^TQ_{k,i}^{12}z\}
\\\nn
                            &\le\sup_{|z|>r}\{-\textstyle{\frac{1}{2}}z^T (Q_{\infty,i}^{22}+\eps_0 I)z+\textstyle{\frac{1}{2}}z^TQ_{k,i}^{22}z+ x^TQ_{k,i}^{12}z\}
\\\nn
                            &\le\sup_{|z|>r}\{\ts{\frac{1}{2}}z^T(Q_{k,i}^{22}-Q_{\infty,i}^{22}) z-\ts{\frac{1}{2}}\eps_0\,z^Tz+ x^TQ_{k,i}^{12}z\}
\\\nn
                            &\le\sup_{|z|>r}\{-\ts{\frac{1}{4}}\eps_0\,z^Tz+x^TQ_{k,i}^{12}z\}
\\\nn
                            &=r\,|Q_{k,i}^{21}x|-\ts{\frac{1}{4}}\eps_0\,r^2\le r\,\ts{\frac{1}{8}}\eps_0\,r-\ts{\frac{1}{4}}\eps_0\,r^2=-\ts{\frac{1}{8}}\eps_0\,r^2,
\end{align}
where the first inequality follows by (iii), the second inequality follows by inspection, and the third inequality follows by the left-hand inequality of \er{prof:thm-limit-3}. Thus, there exists $\bar{r}\ge r_1$ such that
$
\sup_{|z|>\bar{r}}\{f_{k,i}^x(z)\}\le\max_{|z|\le r_0}\{f_{k,i}^x(z)\}.
$
Then, it follows
\begin{align}
\sup_{z\in\R^n}\{f_{k,i}^x(z)\}&=\max\left\{\max_{|z|\le r_0}\{f_{k,i}^x(z)\}, \max_{r_0<|z|\le \bar{r}}\{f_{k,i}^x(z)\}, \sup_{|z|> \bar{r}}\{f_{k,i}^x(z)\}\right\}
\\\nn
                               &=\max\left\{\max_{|z|\le r_0}\{f_{k,i}^x(z)\}, \max_{r_0<|z|\le \bar{r}}\{f_{k,i}^x(z)\}\right\}=\max_{|z|\le \bar{r}}\{f_{k,i}^x(z)\}.
\end{align}
Hence, \er{prof:thm-limit-2} is proved. This, together with the continuity of $\widehat\Psi^i$, implies that the maximizing points $z_k^\ast(x)\doteq \arg\max_{x\in\R^n}\{f_{k,i}^x(z)\}$ exist and are uniformly bounded for $k\ge K$.

Next it is shown that the sequence of functions $f_{k,i}^x$ uniformly converges to $f_{\infty,i}(z)\doteq \widehat{\Psi}^i(z)+ \textstyle{\frac{1}{2}}z^TQ_{\infty,i}^{22}z, \forall z\in\R^n$ on set $B_{\bar{r}}=\{z\in\R^n \,|\,|z|\le \bar{r}\}$. For any $k\in\Z_{>0}$,
\begin{align}
\max_{z\in B_{\bar{r}}}|f_{k,i}^x(z)-f_{\infty,i}(z)|&=\max_{z\in B_{\bar{r}}}| \ts{\frac{1}{2}}z^T(Q_{k,i}^{22}-Q_{\infty,i}^{22})z+x^T Q_{k,i}^{12}\,z|
\\\nn
                                                     &\le\max_{z\in B_{\bar{r}}}| \ts{\frac{1}{2}}z^T(Q_{k,i}^{22}-Q_{\infty,i}^{22})z|+\max_{z\in B_{\bar{r}}}|x^T Q_{k,i}^{12}\,z|
\\\nn
                                                     &=\ts{\frac{1}{2}}\bar{r}^2\,||Q_{k,i}^{22}-Q_{\infty,i}^{22}||_2^2+\bar{r}\,|Q_{k,i}^{21}x|\rightarrow0,
\end{align}
which proves the uniform convergence of the sequence $\{f_{k,i}^x\}_{k=1}^\infty$ to $f_{\infty,i}$ on $B_{\bar{r}}$. \er{prof:thm-limit-1} follows by
\begin{align}
\nn
\lim_{k\rightarrow\infty}\ds{\sup_{z\in\R^n}}\left\{f_{k,i}^x(z)\right\}&=\lim_{k\rightarrow\infty}\max_{z\in B_{\bar{r}}}\left\{f_{k,i}^x(z)\right\}=\max_{z\in B_{\bar{r}}}\lim_{k\rightarrow\infty}\left\{f_{k,i}^x(z)\right\}
=\max_{z\in B_{\bar{r}}}\{f_{\infty,i}(z)\}=\kappa,
\end{align}
where finiteness of $\kappa$ follows by (ii).
\end{proof}

\section{Examples}
\label{sec:exam}

The computational method of Section \ref{sec:comp-method} is illustrated via three examples.

For the purposes of benchmarking, the first example employs a quadratic terminal payoff, and so is a standard LQR problem. The associated value function $W_k$ of \er{eq:value} is computed (over a range of $k\in\Z_{>0}$) via three approaches, namely, (i) via the difference Riccati equation \er{eq:DRE}, (ii) via a grid-based method, involving direct iteration of the dynamic programming equation \er{eq:DPP} on a discretized state space, and (iii) via the max-plus based computational method of Section \ref{sec:comp-method}.
(Note that (ii) represents a standard computational approach to solving a linear regulator problem where the terminal payoff is not quadratic.) The value function computed via (i) is regarded as the actual solution of the LQR problem, for the purposes of comparing the approximation errors obtained in computations (ii) and (iii). This also facilitates the comparison of computation times required to achieve an apriori fixed approximation error via (ii) and (iii), relative to the solution obtained in (i).

The second example examines in further detail the convergence of the max-plus based fundamental solution that underlies the computational method (iii) of Section \ref{sec:comp-method}. In particular, Theorem \ref{thm:conv-circled} is tested. This is independent of the terminal payoff selected.

The third (and final) example considers an infinite horizon linear regulator problem with a non-quadratic terminal payoff. Value functions for the finite and infinite horizon problems are computed using the computational method (iii) of Section \ref{sec:comp-method}.

%%	Example 1.

\subsection{Benchmarking via an LQR problem}
\label{sec:ex-LQR}

With a view to benchmarking the computational method of Section \ref{sec:comp-method}, consider an LQR problem defined as per \er{eq:value} and \er{eq:payoff}, with $\gamma\doteq\sqrt{10}$,
\begin{align}
	A
	& \doteq
	\left[\ba{cc}-0.1&0\\ -0.2&-0.1\ea\right],
	\
	B\doteq\left[\ba{c}0.1\\0.03\ea\right],
	\
	\Phi\doteq\left[\ba{cc}1&0.2\\0.2&2\ea\right],
	\
	\Lambda\doteq\left[\ba{cc}1&0.2\\0.2 & 0.5\ea\right].
	\label{eq:ex-LQR}
\end{align}
The terminal payoff is quadratic, with $\Psi(x) \doteq \demi x^T \Lambda x$ selected in \er{eq:payoff}.

{\em Computation (i):} The value function $W_k$ corresponding to the solution of LQR problem \er{eq:ex-LQR} can be computed via the difference Riccati equation \er{eq:DRE}. The value function $W_{64}$ computed in this way is
\begin{align}
\label{eq:W64}
W_{64}(x)={\frac{1}{2}}x^T\,P_{64}\,x={\frac{1}{2}}x^T\,\left[\ba{cc}1.1016&0.2429\\0.2429&2.0202\ea\right]\,x\,.
\end{align}
For the comparative purposes, $W_{64}$ is assumed to be actual value function \er{eq:value} that solves the LQR problem \er{eq:ex-LQR}.

\newcommand{\cX}		{{\mathscr{X}}}

{\em Computation (ii):}
An approximation $\widehat W_{64}$ of the value function $W_{64}$ of \er{eq:W64} is computed via a grid-based method. In particular, the dynamic programming equation \er{eq:DPP} is iterated directly, without assuming that the value function is quadratic (as would be the case for a non-quadratic terminal payoff). Bounded and discretized state and control spaces $\cX^2$ and $\cW$ are assumed, with
\begin{align}
	\begin{aligned}
	\mathscr{X}
	& \doteq
	[\ba{cc} -\bar x & \bar x \ea]\cap \mathcal{G}_{\delta_{\mathscr{X}}}\,,
	&& \bar x \doteq 3,\, \delta_{\mathscr{X}} \doteq 0.025\,,
	\\
	\mathscr{W}
	& \doteq [\ba{cc} -\bar w & \bar w \ea] \cap\mathcal{G}_{\delta_{\mathscr{W}}}\,,
	&& \bar w = 1,\, \delta_{\mathscr{W}} = 0.1\,,
	\end{aligned}
	\label{eq:grids}
\end{align}
with $\mathcal{G}_\delta \doteq \{ k\, \delta\in\R \, \big|\, k\in\Z \}$. The dynamic programming principle \er{eq:DPP} is approximated by
\begin{align}
	\widehat{W}_{k+1}
	& = \widehat{\mathcal{S}}_1 \widehat{W}_k\,, \quad \widehat{W}_0 = \Psi\,,
	\label{eq:approx-DPP}
\end{align}
where $(\widehat{S}_1\, \phi):\cX^2\rightarrow\R^-$, $(\widehat{S}_1\, \phi)(x) \doteq \sup_{w\in\mathscr{W}} \left\{ \demi \, x^T\, \Phi\, x - \ts{\frac{\gamma^2}{2}} \, |w|^2 + \phi\circ\pi(A x + B w) \right\}$, approximates \er{eq:op-DPP-1} on $\mathscr{X}^2$ via the projection operator $\pi:\R^2\rightarrow\mathscr{X}^2\subset\R^2$,
% (from state space onto the state space grid)
\begin{align}
	\pi(x)
	& = \pi \left( \left[ \ba{c} x_1 \\ x_2 \ea \right] \right)
	\doteq \left[ \ba{c} \tilde\pi(x_1) \\ \tilde\pi(x_2) \ea \right]\,,
	\quad
	\tilde\pi(\xi)
	\doteq -\bar x + \delta_{\mathscr{X}} \left\lfloor \frac{\bar x + \min ( \max(\xi,\, -\bar x),\, \bar x)}{\delta_{\mathscr{X}}}
	\right\rfloor .
	\label{eq:ex-proj}
\end{align}
Figure \ref{fig:err}(a) illustrates the relative error $e_{\widehat{W}_{64}}:\R^2\rightarrow\R_{\ge 0}$ between $\widehat{W}_{64}$ and $W_{64}$ of \er{eq:W64}, where
\begin{align}
	e_\phi(x)\doteq\left|\frac{\phi(x)-W_{64}(x)}{1 + W_{64}(x)}\right|
	\label{eq:ex-rel-error}
\end{align}

{\em Computation (iii):}
An approximation $\widehat{W}_{64}^2$ of the value function $W_{64}$ of \er{eq:W64} is computed via the computational method of Section \ref{sec:comp-method}, using the max-space vector space $\SPone$ of \er{eq:SP-all} with $r\doteq10^3$. Figure \ref{fig:err}(b) illustrates the relative error $e_{\widehat{W}_{64}^2}:\R^2\rightarrow\R_{\ge 0}$, where $e_\star$ is as per \er{eq:ex-rel-error}. There, evaluation of $e_{\widehat{W}_{64}^2}$ is artificially restricted to the bounded grid $[\ba{cc} -\bar x & \bar x \ea ] \cap\mathcal{G}_{0.5}\subset\R^2$ for display purposes only. (Recall that the computational method of Section \ref{sec:comp-method} is not a grid-based method.)

\begin{figure}[h]
\begin{center}
\subfigure[Finite grid method of computation (ii).]{
\psfrag{XXXX}{\scriptsize{$x_2$}}
\psfrag{YYYY}{\scriptsize{$x_1$}}
\includegraphics[width=7cm,height=5cm]{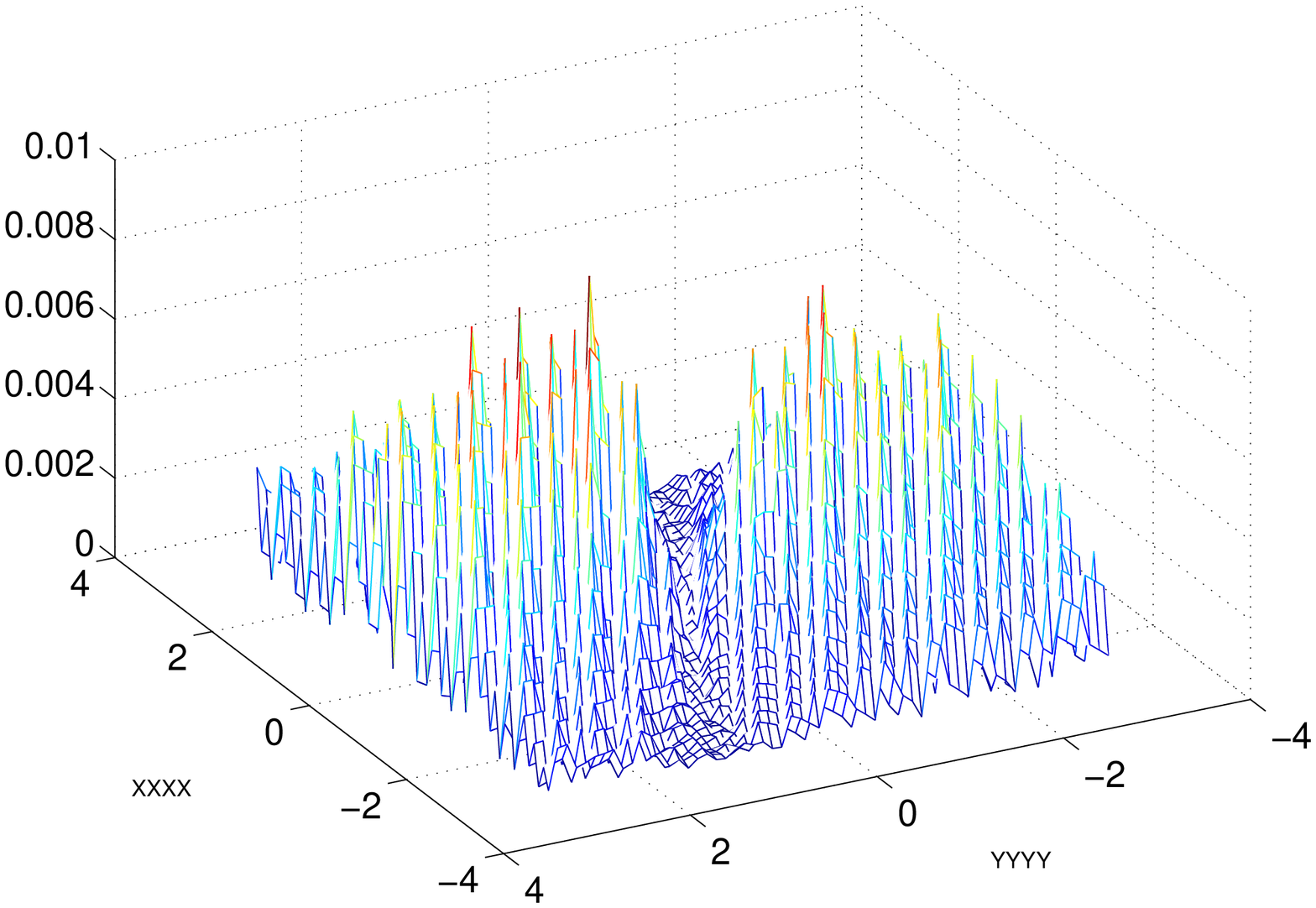}
}
\subfigure[Max-plus method of computation (iii).]{
\psfrag{XX}{\scriptsize{$x_2$}}
\psfrag{YY}{\scriptsize{$x_1$}}
\includegraphics[width=7cm,height=5cm]{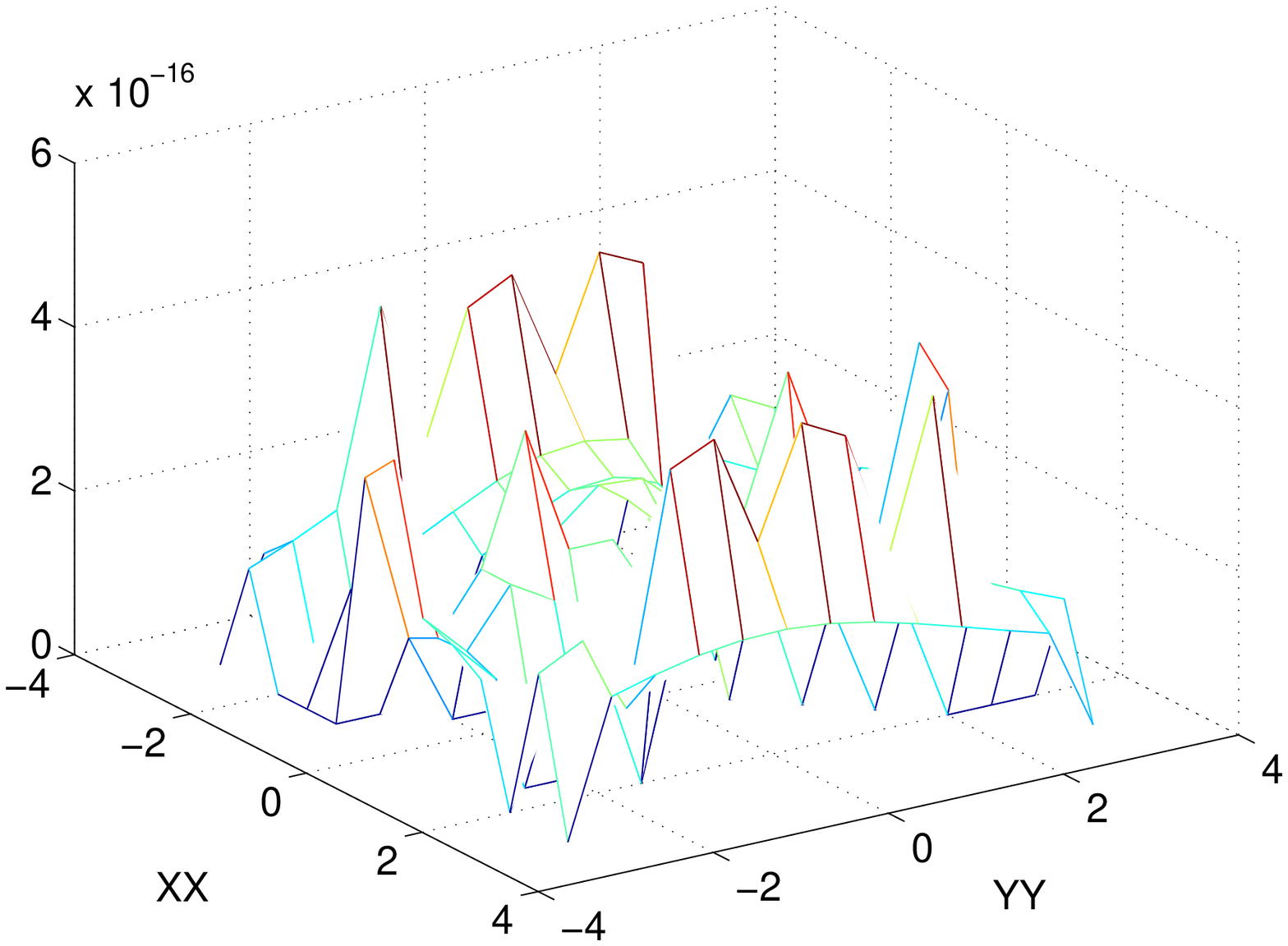}
}
\caption{Relative errors achieved in the approximate solution of an LQR problem (Section \ref{sec:ex-LQR}).}
\label{fig:err}
\end{center}
\end{figure}

\begin{figure}[h]
\begin{center}
\subfigure[Grid-based (DPP iteration) and max-plus method.]{
\psfrag{ddddddddddddddddd}{\scriptsize{control horizon $k$}}
\psfrag{cccccccccccccccc}{\scriptsize{computation time}}
\psfrag{aaaaaaaaaaaaaaaaa}{\scriptsize{grid-based method}}
\psfrag{bbbbbbbbbbbbbbb}{\scriptsize{max-plus method}}
\includegraphics[width=7cm,height=5cm]{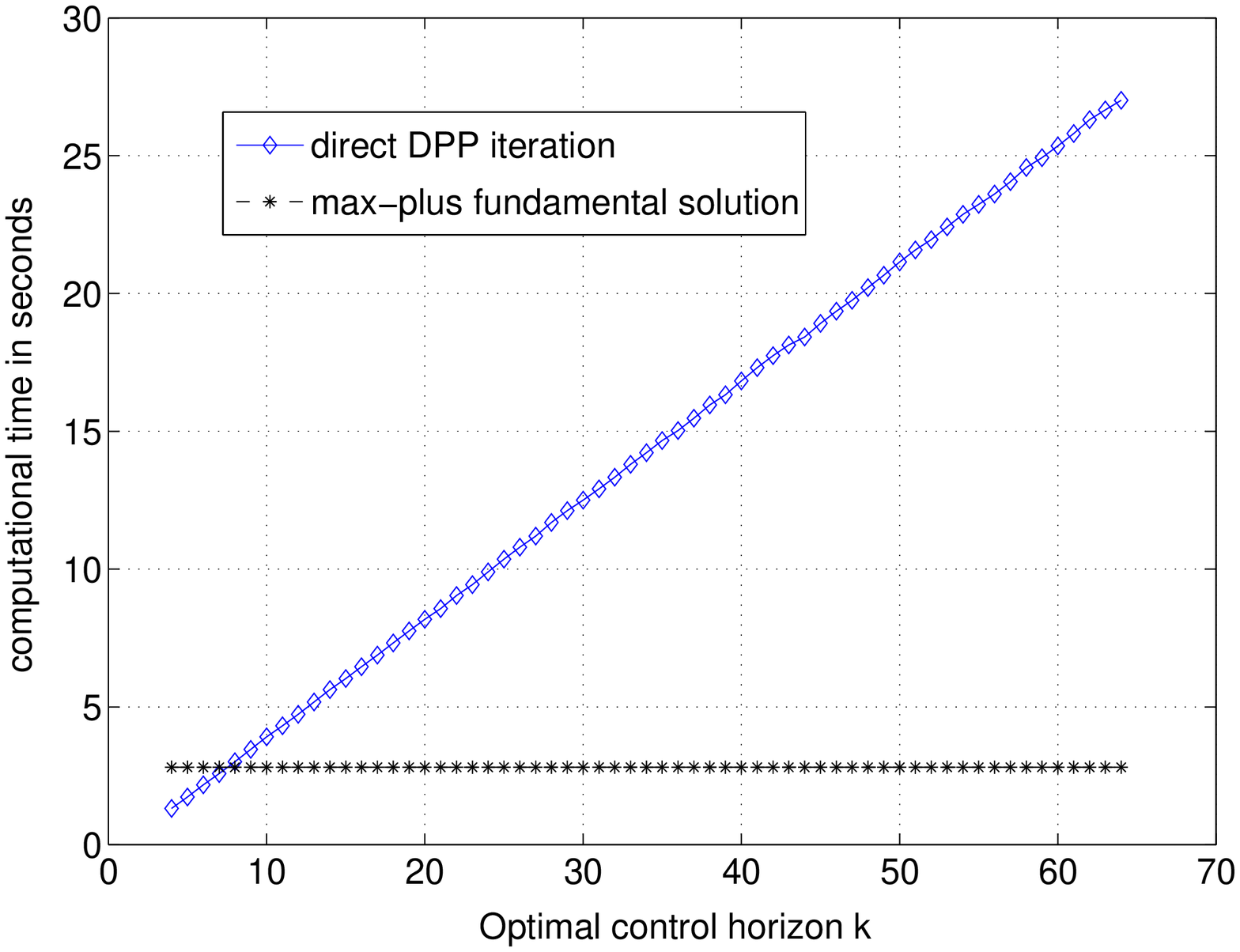}
}
\subfigure[Max-plus method of computation (iii).]{
\psfrag{control horizon k}{\scriptsize{control horizon $k$}}
\psfrag{computational time in seconds}{\scriptsize{computation time}}
\includegraphics[width=7cm,height=5cm]{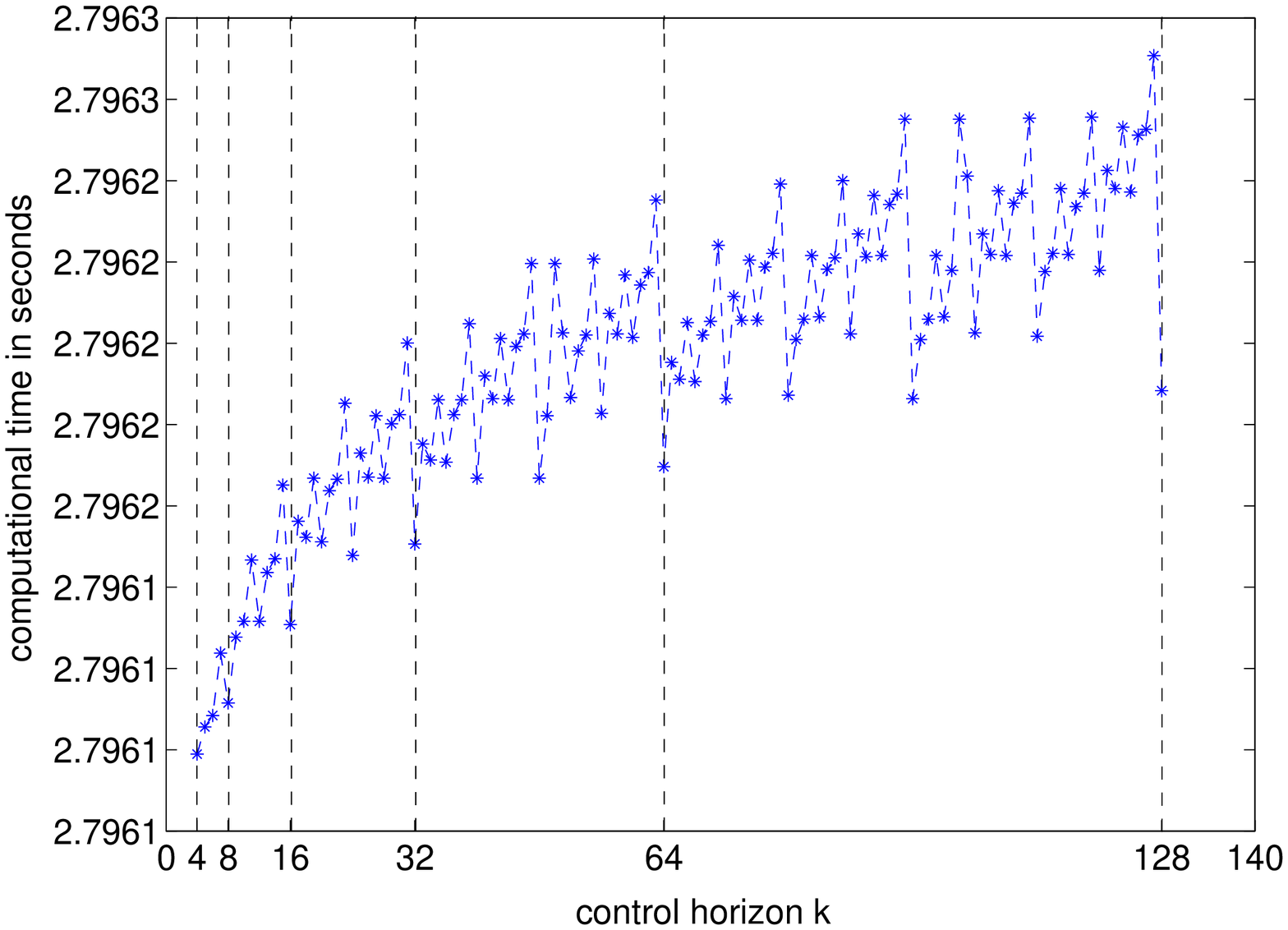}
}
\caption{Computation times achieved in the approximate solution of an LQR problem (Section \ref{sec:ex-LQR}).}
\label{fig:time}
\end{center}
\end{figure}

{\em Error comparison:} By comparison of Figures \ref{fig:err}(a) and (b), it is evident that the max-plus based computation (iii) achieves a significantly smaller relative error than the direct dynamic programming computation (ii) for the same time horizon $k=64$. Indeed, the relative error of computation (iii) is of the order of the machine epsilon for the Dell laptop used. This is attributable to the matrix operations involved in propagating the matrices $\Theta_{k,1}$ in step {\dthree} of the method, and to approximations in the dual / primal operations of steps {\done}, {\dtwo} and {\dfour}, {\dfive}. Meanwhile, the much larger errors observed in computation (ii) are due largely to the state space projection operator $\pi$ of \er{eq:ex-proj} associated with the finite grid employed.

{\em Computation time comparison:} In order to compare computation times of the grid-based computation (ii) and the max-plus based computation (iii), the respective computations of $\widehat{W}_k$ and $\widehat{W}_k^2$ are repeated for all $k\in[1,128]\cup\Z_{>0}$. Time index doubling is employed in the latter computation (iii) to demonstrate the speed-up achievable via the max-plus based computation.
Figure \ref{fig:time}(a) illustrates an overlay of the computation times for computations (ii) and (iii) on the same axes. This demonstrates an approximately linear growth in computation time with time index $k$ for the grid-based method of (ii), and an approximately constant computation time for the max-plus method of (iii). A definitive computational advantage is evident in the max-plus case for all but small time indices. In examining this computational advantage further, Figure \ref{fig:time}(b) illustrates that the computation time for the max-plus based method of (iii) does in fact vary with the time index $k$. This computation time maybe approximated by
$T_k=\hat{t}+t_k$. Here, $\hat{t}$ denotes the time used to compute the dual of terminal payoff in Step {\dtwo}, the matrix $Q_{k,1}=\Gamma^1(\Theta_{k,1})$ in Step {\dfour}, and the value function $\widehat{W}_{k}^2$ in Step {\dfive}. $\hat{t}$ is independent of control horizon $k$, and is $2.7961$ seconds here. $t_k$ denotes the total time used to propagate the Hessian $\Theta_{1,2}$ to Hessian $\Theta_{k,2}$ in Step {\dthree}. The non-monotone behaviour observed in the growth of this computation time is due to the time index doubling employed in the computation (iii). In order to understand this behaviour, it is useful to employ a binary (base-$2$) representation for the time index $k$, with
% \begin{align}
%	k
%	& = b_1\, 2^{p_k-1} + b_2\, 2^{p_k-2} + \cdots + b_{p_k}\, 2^0\,,\quad
%	b_j\in\{0,1\}, \ j\in[1,p_k]\cap\Z_{>0}, \ p_k \doteq 1 + \lceil \log_2 k \rceil\,.
%	\nn
% \end{align}
\begin{align}
	k
	& = \sum_{j=0}^{m_k-1} b_{j}\, 2^{j}
	= (b_{m_{k}-1} \cdots b_2 b_1 b_0)_2 \,,\quad
	b_j\in\{0,1\}, \ j\in[0,m_k-1]\cap\Z, \ m_k \doteq 1+ \lfloor \log_2 k \rfloor\,,
	\nn
\end{align}
in which $m_k\in\Z_{>0}$ denotes the minimum number of ``bits'' required for the base-$2$ representation. By definition, $b_{m_k-1} = 1$ for all $k\in\Z_{\ge 1}$. Using this notation, $n_k\doteq\sum_{j=0}^{m_{k}-1} b_j$ denotes the number of non-zero ``bits'' $b_j$ in this representation of $k$. Let $\tau$ denote the time required to perform the matrix operation $\circledast$ of \er{eq:B-update} employed in the propagation step {\dthree}. (Recall that $\circledast$ is central to the propagation of the Hessian $\Theta_{k,1}$ of the kernel $B_{k,1}$ of the max-plus integral operator $\mathcal{B}_{k,1}$, that is itself central to max-plus based computation (iii) -- see \er{eq:B-update}, \er{eq:quad-fund-dual}, and \er{eq:mp-semigroup} respectively.) Computation of $\widehat{W}_{k}^2 = W_k$ requires $m_k -1$ time index doubling steps to increase the time index from $1$ up to $2^{m_k-1}$, plus an additional $n_k - 1$ time index ``sub-doubling'' steps to further increase the time index from $2^{m_k-1}+1$ up to $k$. For example, a time index of $k=50$ has a $m_{50} = 1 + 5 = 6$ bit binary representation $50=(110010)_2$, with $n_{50} = 3$ non-zero bits, implying that $m_{50} - 1 = 5$ time index doubling steps plus $n_{50} - 1 = 2$ sub-doubling steps are required. Hence, the sequence of these $\circledast$ steps used to compute Hessian $\Theta_{50,1}$ from $\Theta_{1,1}$ (i.e. corresponding to the value function $W_{50}$) is then
\begin{align}
	\underbrace{
	\ba{cccccccccccc}
	&&&&&&&&&&&
	\\
	&&&&&&&&&&&
	\\
	\Theta_{1,1}
	& \rightrightarrows
	& \Theta_{2,1}
	& \rightrightarrows
	& \Theta_{4,1}
	& \rightrightarrows
	& \Theta_{8,1}
	& \rightrightarrows
	& \Theta_{16,1}
	& \rightrightarrows
	& \Theta_{32,1}
	\ea
	\hspace{-3mm}
	}_{\text{Doubling steps}}
	\underbrace{
	\ba{cccc}
		&
		\Theta_{16,1} && \Theta_{2,1}
		\\
		&
		\downarrow && \downarrow
		\\
		\rightarrow
		&
		\Theta_{48,1}
		& \rightarrow
		& \Theta_{50,1}
	\ea
	}_{\text{Sub-doubling steps}}
	\nn
\end{align}
where each arrow corresponds to an incoming argument to a matrix $\circledast$ operation. In general, as each doubling or sub-doubling step requires an application of one $\circledast$ operation (taking time $\tau$ per operation), the total computation time needed to compute $\Theta_{k,2}$ may be approximated by
\begin{align}
	t_k
	& \doteq ((m_k - 1) + (n_k-1)) \, \tau \le 2\, \tau\, (m_k - 1) = 2 \, \tau\, \lfloor \log_2 k \rfloor\,.
	\label{eq:comp-time}
\end{align}
Hence, the non-monotone growth of the computation time $t_k$ observed in Figure \ref{fig:time}(b) is due to the dependence of $t_k$ on $k$ above in \er{eq:comp-time}. This computation time is independent of the terminal payoff selected (whether quadratic or non-quadratic).

In this specific implementation of the propagation $\Theta_{k,1}$ in Step {\dthree}, $n_k$ matrices $\Theta_{2^j,1}$ for $j\in[0,n_k-1]\cap \Z$ such that $b_j=1$ must be stored in order to perform the ``sub-doubling'' steps. In the worst case, $n_k=m_k=1+\lfloor \log_2{k}\rfloor$ steps are required (where $k=2^{m_k}-1$). In order to avoid the attendant increase in memory required to store all $n_k$ matrices $\Theta_{2^j,1}$, $j\in[1,n_k]\cap\Z$, some matrices (for example, those ones with smaller $j$) need not be stored. Instead, they can be recomputed from $\Theta_{1,1}$ using the $\circledast$ matrix operation. In the worst case (for computation time), all such matrices used in the ``sub-doubling'' steps can be recomputed. The worst-case total time required for computing $\Theta_{k,1}, k=2^{m_k}-1$ using such a scheme is given by
$$
t_k= \left(\sum_{j=1}^{m_k-1}j+(m_k-1)\right)\,\tau=(m_k-1)\left({\frac{m_k}{2}}+1\right)\,\tau=\lfloor \log_2 k \rfloor\,\left(\frac{\lfloor \log_2 k \rfloor+3}{2}\right)\,\tau.
$$
It may be noted that for current computational platforms and typical linear regulator problems, this worst-case recomputation is not required, as the memory usage remains relatively small.

\subsection{Convergence of the max-plus based fundamental solution on $\SPthree$}
For infinite horizon linear regulator problems, convergence of a sequence of Hessians $\{\Theta_{2^{k},i}\}_{k=1}^\infty$, $i\in\{1,2,3\}$, generated via time index doubling (for example) is crucial to the application of the computational method of Section \ref{sec:comp-method}. Theorem \ref{thm:conv-circled} states that this sequence is convergent, under specific conditions. The purpose of this example is to test the conditions of that theorem. To this end, consider a linear regulator problem defined as per \er{eq:value} and \er{eq:payoff}, with
\begin{align}
	A & \doteq \left[\ba{cc}-0.2&0.1\\-0.15 &0 \ea\right],
	\quad
	B \doteq \left[\ba{cc}1&0\\0&1 \ea\right],
	\quad
	\Phi \doteq  \left[\ba{cc}0.6&0\\0&0.2\ea\right],
	\quad
	\gamma = \sqrt{8}\,.
	\label{eq:ex-convergence}
\end{align}
(Note that convergence or otherwise of the aforementioned sequence is independent of the terminal payoff $\Psi$. Hence, $\Psi$ is not specified in this example.) The sequence of interest, generated by time index doubling in computing the fundamental solution in $\mathscr{B}_r^3$, is
\begin{align}
	\Theta_{2^{k+1},3}
	& = \Theta_{2^k,3}\circledast\Theta_{2^k,3}\ , \quad\quad k\in\Z_{\ge0}\,,
	\label{eq:sequence-theta-k-3}
\end{align}
initialized with $\Theta_{1,3}\doteq-Q_{1,3}$ where $Q_{1,3}$ is given by \er{eq:Q-1-3}. In order to verify the convergence of this sequence via Theorem \ref{thm:conv-circled}, define
\begin{align}
	\sigma
	&\doteq \lambda_{\max}(\Theta_{1,3}^{12}\Theta_{1,3}^{21})
	= \lambda_{\max}(\gamma^4A^T(BB^T)^{-2}A)
	= 4.4321\,,
	\nn\\
	\lambda
	& \doteq \lambda_{\min}(\Theta_{1,3}^{11}+\Theta_{1,3}^{22})
	= \lambda_{\min}(-\Phi+\gamma^2A^T(BB^T)^{-1}A+\gamma^2(BB^T)^{-1})
	= 7.7297\,.
	\nn
\end{align}
These definitions imply that condition \er{eq:ineq-Omega} of Theorem \ref{thm:conv-circled} holds for $\Omega = \Theta_{1,3}$. The remaining condition \er{eq:rhosiglam} of Theorem \ref{thm:conv-circled} holds if there exists $\bar\rho>\sqrt{\sigma}$ such that $f(\bar\rho)>0$, where
\begin{align}
	f(\rho)
	& \doteq \lambda-\rho-2\rho^{-1}\sigma(1-\rho^{-2}\sigma)^{-1}\,.
	\label{eq:ex-function-f}
\end{align}
This may readily be verified via some simple working, or graphically via Figure \ref{fig:conv}(a). (For example, select $\bar\rho \doteq 4$.)
Hence, the conditions of Theorem \ref{thm:conv-circled} hold, so that the matrix sequence \er{eq:sequence-theta-k-3} must converge to the matrix limit $\Theta_{\infty,3} = \text{diag}(\Theta_{\infty,3}^{11}, \, \Theta_{\infty,3}^{22} )$. This convergence may be observed by enumerating the sequence for sufficiently large $k$. Figure \ref{fig:conv}(b) illustrates the sequences $\{\sigma_{2^k}\}$ and $\{\lambda_{2^k}\}$ of \er{eq:seqa}, and the sequences $\{\sigma_{2^k}'\}$ and $\{\lambda_{2^k}'\}$ defined by
\begin{align}
	\sigma'_{2^{k-1}}
	\doteq \lambda_{\max}(\Theta_{2^{k-1},3}^{12}\Theta_{2^{k-1},3}^{21})\,,
	\quad
	\lambda'_{2^{k-1}}
	\doteq \lambda_{\min}(\Theta_{2^{k-1},3}^{11}+\Theta_{2^{k-1},3}^{22})\,.
	\nn
\end{align}
These sequences may be observed to be monotone, as expected. The aforementioned limit $\Theta_{\infty,3}$ may be computed as
$$
	\Theta_{\infty,3}
	= \left[\ba{cc|cc}
		-0.6313&0.0135&0.000&0.0000
		\\
		0.0135&-0.2069&0.0000&0.0000
		\\\hline
		0.0000&0.0000&7.5921&-0.2502
		\\
		0.0000&0.0000&-0.2502&7.8072
	\ea\right].
$$

\begin{figure}[h]
\begin{center}
\subfigure[Function $f$ of \er{eq:ex-function-f}.]{
\psfrag{xx}{\scriptsize{$\rho$}}
\psfrag{ffff}{\scriptsize{$f(\rho)$}}
\includegraphics[width=7cm,height=5cm]{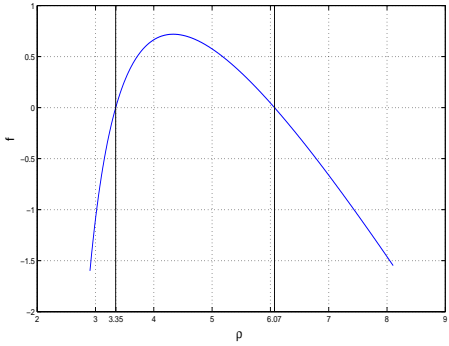}
}
\subfigure[Convergence of sequences $\sigma_{2^k},\sigma'_{2^k},\lambda_{2^k},\lambda'_{2^{k}}$.]{
\psfrag{time k}{\scriptsize{control horizon $k$}}
\psfrag{xxxxxxxx}{\scriptsize$\sigma'_{2^{k-1}}$}
\psfrag{yyyyyyyy}{\scriptsize$\sigma_{2^{k-1}}$}
\psfrag{zzzzzzzz}{\scriptsize$\lambda'_{2^{k-1}}$}
\psfrag{qqqqqqqq}{\scriptsize$\lambda_{2^{k-1}}$}
\includegraphics[width=7cm,height=5cm]{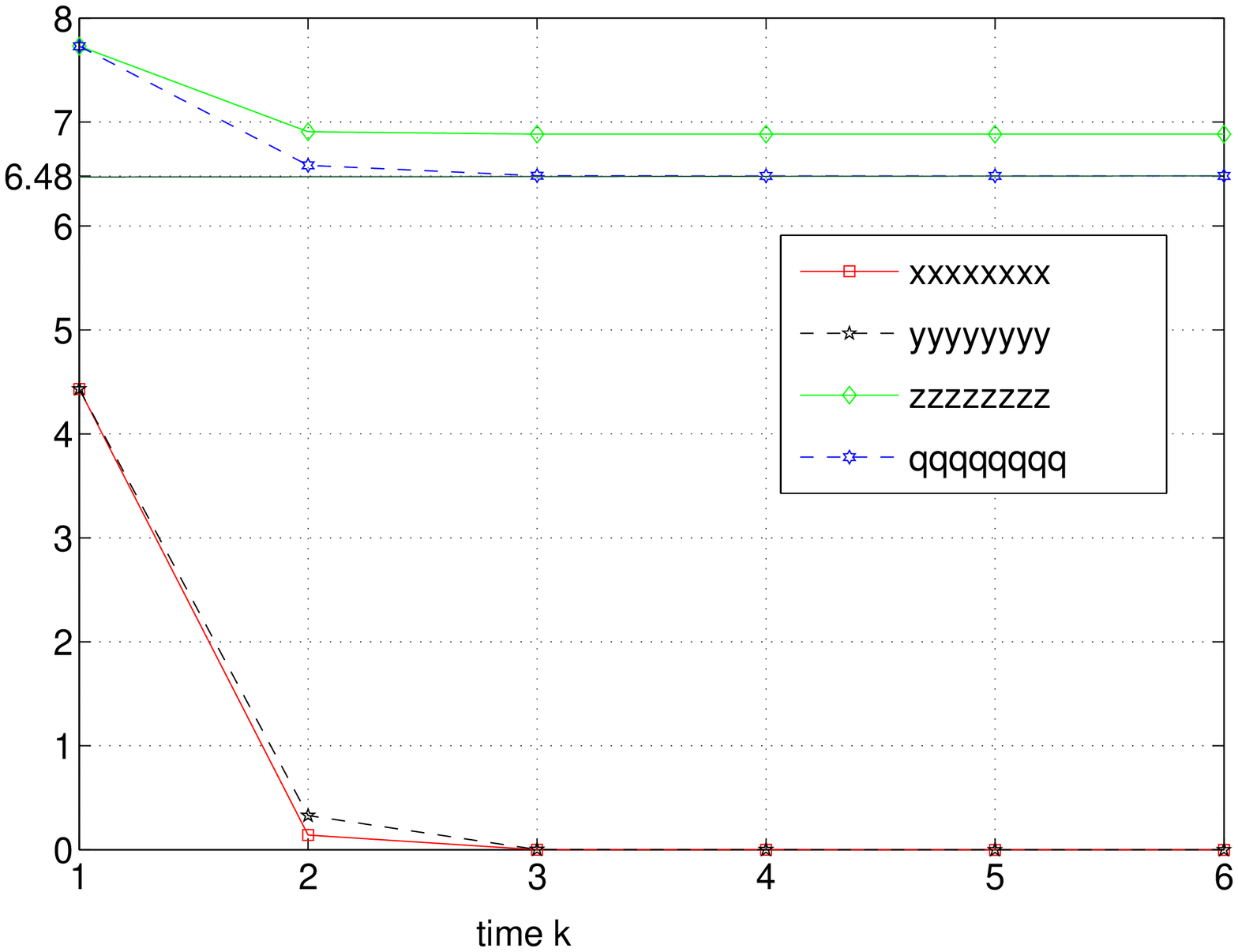}
}
\caption{Convergence of the max-plus based fundamental solution on $\mathscr{B}_r^3$.}
\label{fig:conv}
\end{center}
\end{figure}

%%%%%%%%%%%%%%%%%%%%%%%%%%%%%%%%%%%%%%%%%%%%%%

\subsection{Infinite horizon linear regulator problem with non-quadratic payoff on $\SPtwo$}
In order to demonstrate that the value function of infinite horizon linear regulator problem is quadratic with an offset according to Theorem 4.5, consider the linear regulator problem with non-quadratic payoff given by
\begin{align}
	A \doteq & \left[\ba{cc}-0.12&0\\0.1 &0.15 \ea\right],
	\quad
	B \doteq \left[\ba{cc}-0.2\\0.1 \ea\right],
	\quad
	\Phi \doteq \left[\ba{cc}3&-1.4\\-1.4&2.4\ea\right],
    \quad
    \gamma\doteq 2,
	\nn\\
	\Psi(x)& = \Psi( [ \ba{cc} x_1 & x_2 \ea ]^T ) \doteq 3|x_2+1|\,|\sin(x_1-1)|\,.
	\nn
\end{align}	
The max-plus based fundamental solution on $\SPtwo$ is employed, with $M \doteq \left[\ba{cc}10 &0\\0&10 \ea\right]$. 
Figure \ref{fig:terminal} shows the non-quadratic terminal payoff $\Psi$ and its max-plus dual $\widehat{\Psi}$. Note that $\Psi$ and $\widehat\Psi$ appear similar since a relatively big $M$ is used. Recall that $\Psi$ and $\widehat\Psi$ will be the same when $M\rightarrow\infty\,I$ which corresponds to the duality in $\SP{3}$.

\begin{figure}[h]
\begin{center}
\subfigure[Terminal payoff $\Psi$.]{
\psfrag{xxxx}{\scriptsize{$x_2$}}
\psfrag{yyyy}{\scriptsize{$x_1$}}
\psfrag{zzzzzz}{\scriptsize{$\Psi$}}
\includegraphics[width=7cm,height=5cm]{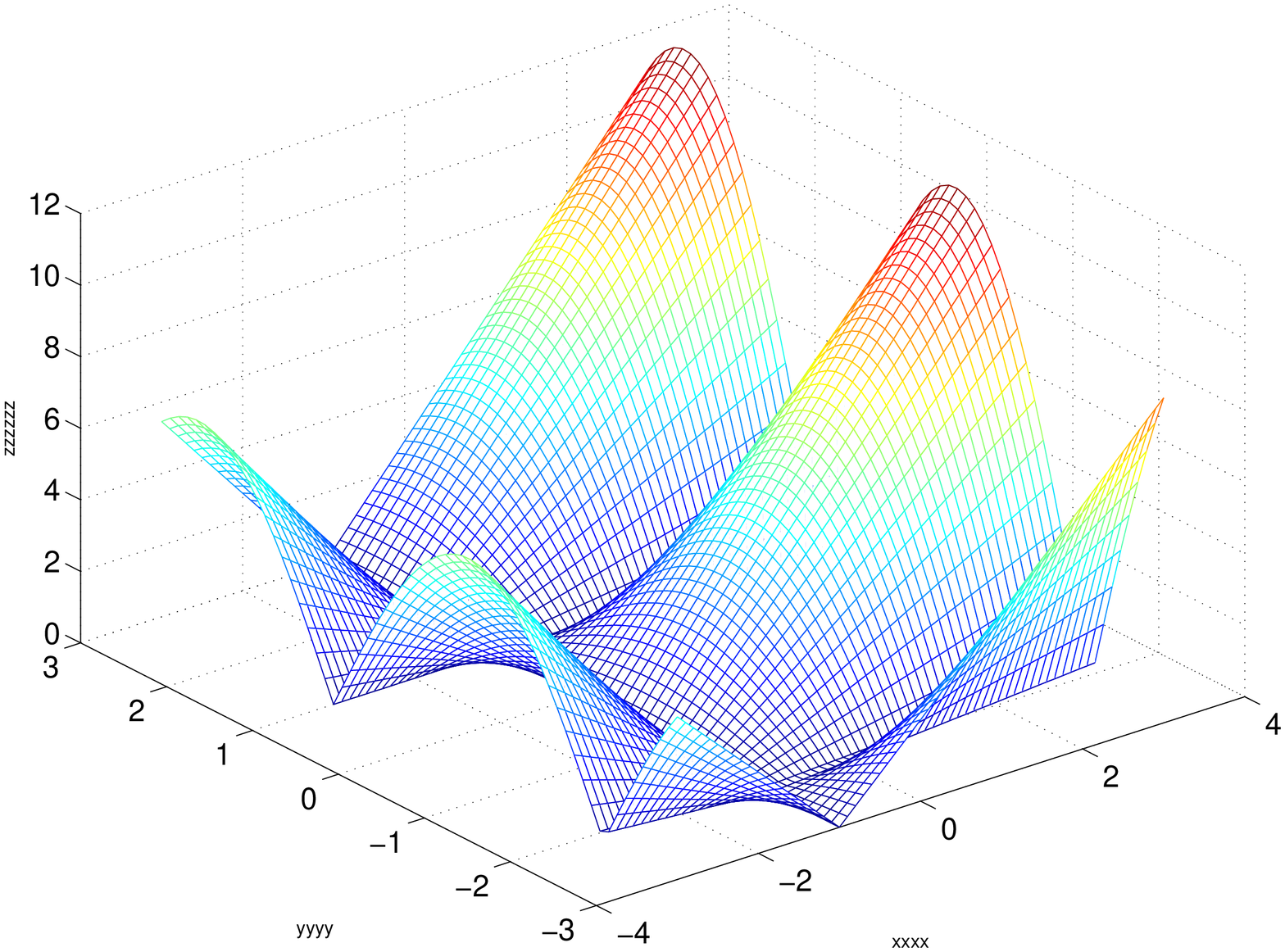}
}
\subfigure[Max-plus dual  $\widehat\Psi$ of the terminal payoff $\Psi$.]{
\psfrag{xxxx}{\scriptsize {$x_2$}}
\psfrag{yyyy}{\scriptsize{$x_1$}}
\psfrag{zzzzzz}{\scriptsize{$\widehat\Psi$}}
\includegraphics[width=7cm,height=5cm]{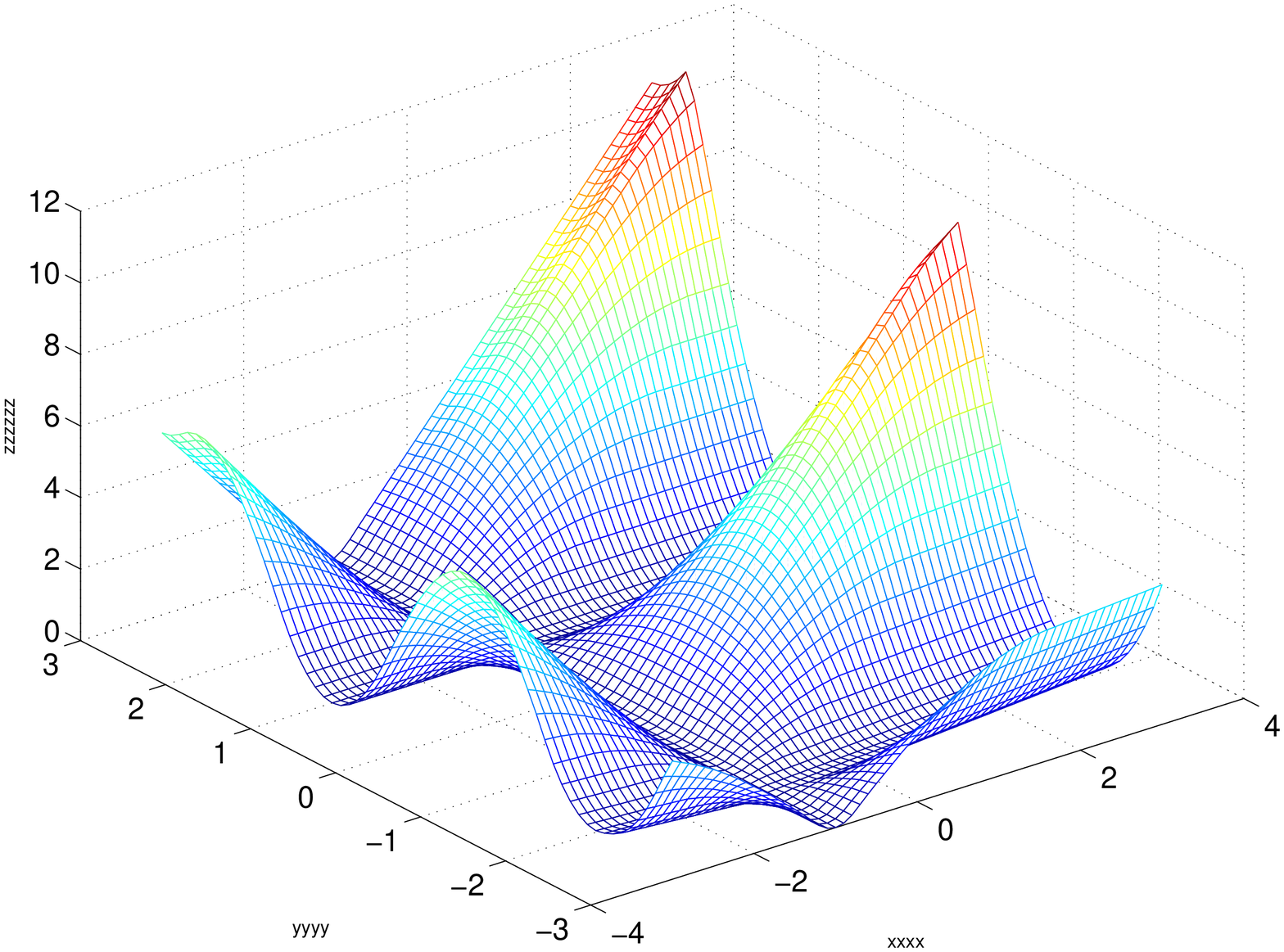}
}
\caption{Non-quadratic terminal payoff $\Psi$ and its max-plus dual $\widehat\Psi$.}
\label{fig:terminal}
\end{center}
\end{figure}

The convergence of the sequence $\{\Theta_{k,2}\}_{k=1}^\infty$ is essential to compute the value function of the infinite horizon linear regulator problems. According to Theorem 4.2, Theorem 4.3, and Theorem 4.4,  $\Theta_{k,2}\rightarrow \Theta_{\infty,2}=\left[\ba{cc} \Theta_{\infty,2}^{11}&0\\0&\Theta_{\infty,2}^{22} \ea\right]$ if the inequalities in Theorem 4.2 are satisfied for $\Theta_{1,2}$. This can be tested similarly to the example in Section 5.2. In particular, here $\Theta_{1,2}$ is computed by (37)
$$
\Theta_{1,2}= \left[\ba{cc|cc}
 -2.0555   & 1.0036 &   0.8266 &  -0.8816
 \\
    1.0036  & -1.6630 &   0.0497 &  -1.3155
    \\\hline
    0.8266  &  0.0497 &   9.1975 &   0.3607
    \\
   -0.8816 &  -1.3155 &   0.3607 &  10.0522
   \ea\right].
$$
Take 
\begin{align}
\nn
\hat\sigma\doteq \lambda_{\text{max}}(\Theta_{1,2}^{12}\Theta_{1,2}^{21})=2.8054,
\quad
\hat\lambda\doteq\lambda_{\text{min}}(\Theta_{1,2}^{11}+\Theta_{1,2}^{22})=  6.2655.
\end{align}
From Lemma 4.1, the conditions in Theorem 4.2 will be satisfied if there exists a $\hat\rho>\sqrt{\hat\sigma}$ such that $\hat f(\hat\rho)>0$, where the function $\hat f$ is
\begin{align}
\label{eq:fun_hatf}
\hat f(\rho)\doteq \hat\lambda-\rho-2{\rho}^{-1}\hat\sigma(1-\rho^{-2}\hat\sigma)^{-1}
\end{align}
as shown in Panel (a) in Figure \ref{fig:fun_hatf}. By observation, $\hat f(\hat\rho)>0$ for any $2.6249<\hat\rho<5.0049$. Thus, according to Lemma 5.1 and Theorem 5.2, the sequences defined by 
\begin{align}
	\hat\sigma'_{2^{k-1}}
	\doteq \lambda_{\max}(\Theta_{2^{k-1},2}^{12}\Theta_{2^{k-1},2}^{21})\,,
	\quad
	\hat\lambda'_{2^{k-1}}
	\doteq \lambda_{\min}(\Theta_{2^{k-1},2}^{11}+\Theta_{2^{k-1},2}^{22})\,.
	\nn
\end{align}
converge as shown in Panel (b) of Figure  \ref{fig:fun_hatf}. Hence, the sequence $\{\Theta_{k,2}\}_{k=1}^\infty$ converges to a block diagonal matrix as $k\rightarrow\infty$ which is computed as
\begin{align}
\Theta_{\infty,2}= \left[\ba{cc|cc}
-2.2859   & 0.8275  &  0.0000 &   0.0000
\\
    0.8275 &  -1.8835  &  0.0000 &   0.0000
    \\\hline
    0.0000  &  0.0000  &  9.0986  &  0.4467
\\
    0.0000 &   0.0000 &   0.4467 &   9.7773
    \ea\right]
\end{align}
Consequently, the $Q_{\infty,2}=\text{diag}(M(\Theta_{\infty,2}^{11}+M)^{-1}M-M,-\Theta_{\infty,2}^{22})$ is 
\begin{align}
Q_{\infty,2}= \left[\ba{cc|cc}
    3.1067  & -1.3362 &  0.0000  & 0.0000
    \\
   -1.3362  &  2.4568 &  0.0000 &  0.0000
   \\\hline
   0.0000 &  0.0000 &  -9.0986 &  -0.4467
   \\
   0.0000 &  0.0000  & -0.4467 &  -9.7773
    \ea\right]
\end{align}

\begin{figure}[h]
\begin{center}
\subfigure[The function $\hat{f}$ of \er{eq:fun_hatf}.]{
\psfrag{xxxx}{\scriptsize {$\rho$}}
\psfrag{ffff}{\scriptsize{$\hat f(\rho)$}}
\includegraphics[width=7cm,height=5cm]{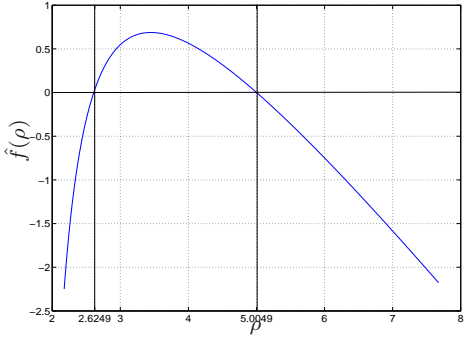}
}
\subfigure[Convergence of various sequences.]{
\psfrag{xxxxxxxxxxxxxxx}{\scriptsize{control horizon $k$}}
\psfrag{aaaaaaaaaaaa}{\scriptsize$\hat{\sigma}'_{2^{k-1}}$}
\psfrag{bbbbbbbbbbbb}{\scriptsize$\hat{\sigma}_{2^{k-1}}$}
\psfrag{ccccccccccccc}{\scriptsize$\hat{\lambda}'_{2^{k-1}}$}
\psfrag{dddddddddddd}{\scriptsize$\hat{\lambda}_{2^{k-1}}$}\includegraphics[width=7cm,height=5cm]{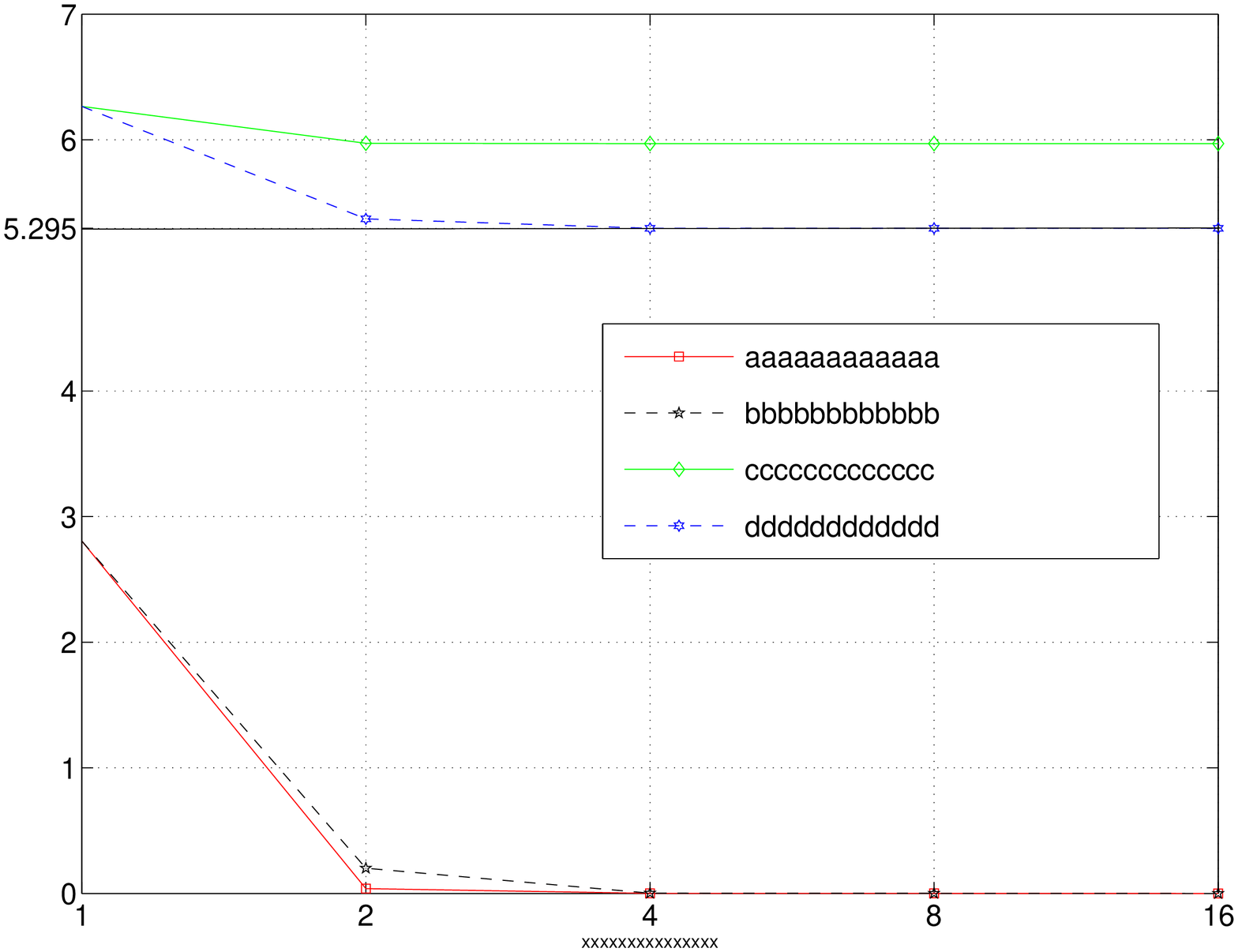}
}
\caption{Convergence of the max-plus fundamental solution $\Theta_{k,2}$.}
\label{fig:fun_hatf}
\end{center}
\end{figure}

It can be verified that $Q_{\infty,2}^{22}<0$ since the eigenvalues of $Q_{\infty,2}^{22}$ are $-9.990$ and $-8.8770$. It is also noted that the terminal payoff $\Psi$ (hence its dual) is oscillating on $x_1$ and linear on $x_2$. Thus, the conditions on Theorem 4.5 (equation (66)) is satisfied. Consequently, the infinite horizon value function $W_\infty$ is quadratic with an offset as given by equation \er{eq:infty-value}. The offset $\kappa$ is computed as 
$$
\kappa\doteq\max_{z\in\R^n}\left\{\widehat\Psi(z)+\frac{1}{2}z^T Q_{\infty,2}^{22}z\right\}= 2.5785.
$$
The value function $W_{\infty}$ is shown in Panel (a) of Figure \ref{fig:val-nonl}. To verify that $W_\infty$ is indeed quadratic, an approximation $\widetilde{W}_\infty$ is computed using the grid based method similar to example 1 in Section 5.1. The relative error defined by 
\begin{align}
\label{eq:err-inf}
	{e}_{\widetilde{W}_\infty}(x)\doteq\left|\frac{\widetilde{W}_\infty(x)-W_{\infty}(x)}{1 + {W}_{\infty}(x)}\right|
\end{align}
is shown in Panel (b) in Figure \ref{fig:val-nonl}. A small relative error  verifies the developed max-plus computational method.
\begin{figure}[h]
\begin{center}
\subfigure[The infinite horizon value function $W_{\infty}$.]{
\psfrag{xxxx}{\scriptsize$x_1$}
\psfrag{yyyy}{\scriptsize$x_2$}
\psfrag{zzzzzzzz}{\scriptsize$W_\infty$}
\includegraphics[width=7cm,height=5cm]{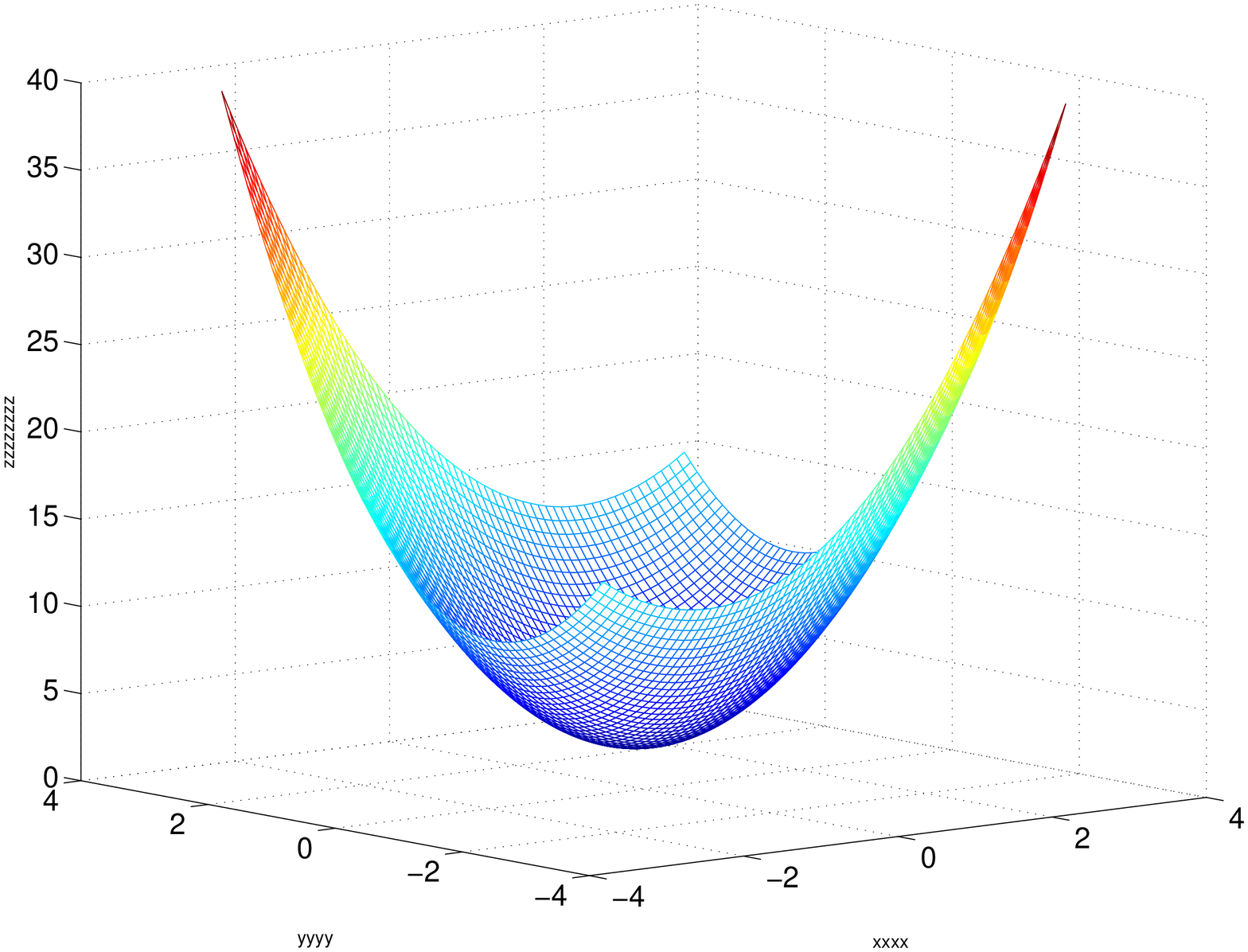}}
\subfigure[The error $e_{\widetilde{W}_\infty}$ for grid based method.]{
\psfrag{xxxx}{\scriptsize$x_1$}
\psfrag{yyyy}{\scriptsize$x_2$}
\includegraphics[width=7cm,height=5cm]{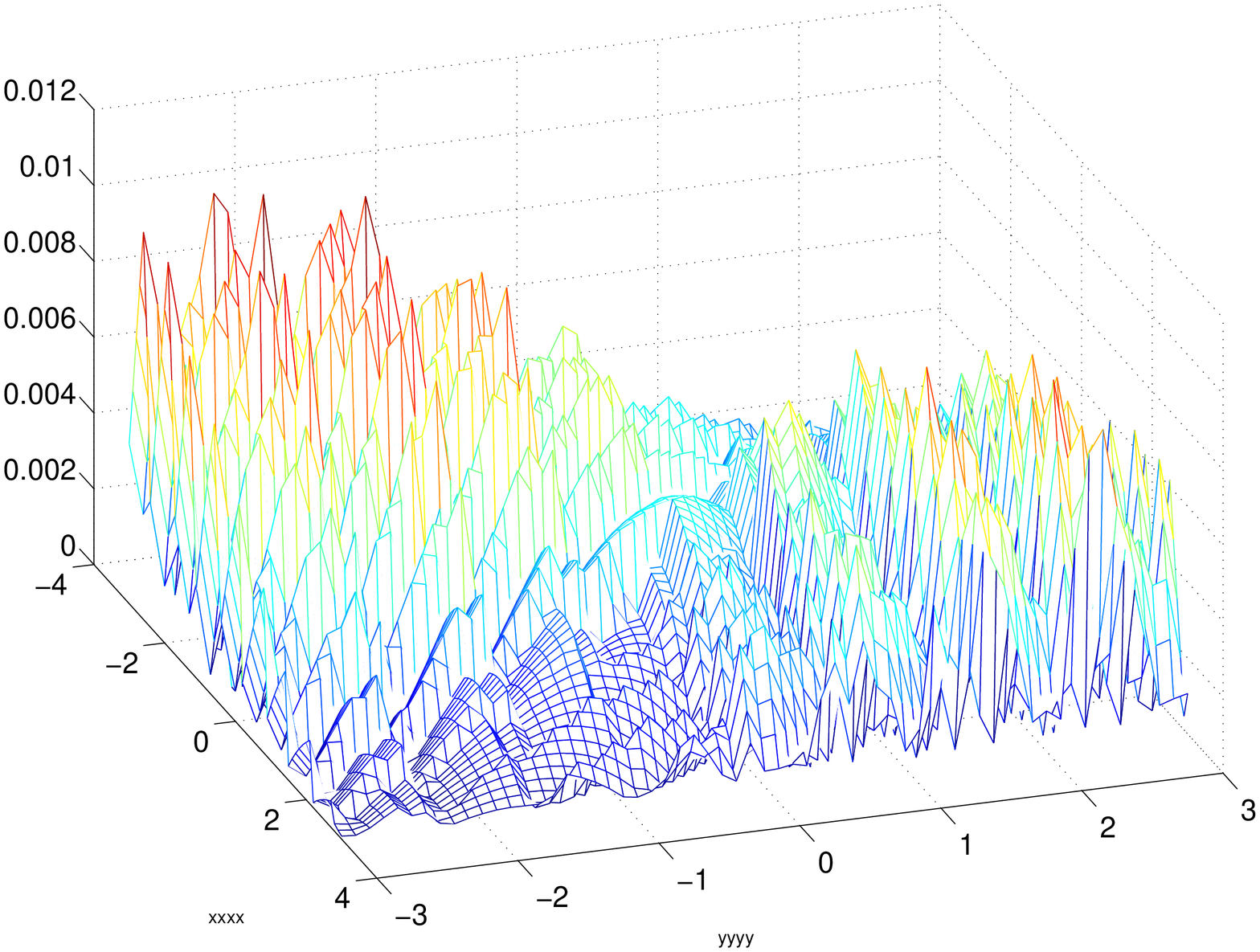}
}
\caption{The infinite horizon value function and relative error \er{eq:err-inf} for grid based method.}
\label{fig:val-nonl}
\end{center}
\end{figure}

\section{Conclusions}
\label{sec:conc}

An efficient computational method is developed for solving a class of discrete time linear regulator problems employing a non-quadratic terminal payoff. Max-plus linearity of the corresponding dynamic programming evolution operator is exploited to obtain a max-plus based solution from which the associated value function may be computed conveniently for any non-quadratic terminal payoff. The computation of the max-plus based fundamental solution is reduced to a sequence of matrix iterations which can be computed efficiently and accurately. A sufficient condition for the convergence of the finite horizon value function to the corresponding infinite horizon value function is presented. This convergence result generalizes the well-known convergence results of difference Riccati equations. Numerical examples are given to demonstrate the performance of the proposed method.

\bibliographystyle{plain}
\bibliography{L2}

\end{document}